\documentclass[reqno,centertags,a4paper,11pt]{amsart}
\usepackage{amssymb}
\usepackage{amsthm}
\usepackage{eucal}
\usepackage{color, soul, xcolor}
\soulregister\cite{7}
\soulregister\ref{7}
\soulregister\pageref7
\usepackage[english]{babel}
\usepackage{graphicx}
\usepackage{subfig}
\usepackage{hyperref}

\usepackage{lineno}

\def\medno{\medskip \noindent}

\newcommand{\R}{\mathbb{R}}

\def\ra{\right>}
\def\la{\left<}

\newcommand{\p}{\partial}
\newcommand{\rr}{\mathbb{R}}

\renewcommand{\p}{\partial}

\DeclareMathOperator{\sech}{sech}
\numberwithin{equation}{section}
\newtheorem{theorem}{Theorem}[section]
\newtheorem{proposition}[theorem]{Proposition}
\newtheorem{remark}[theorem]{Remark}
\newtheorem{lemma}[theorem]{Lemma}

\begin{document}
\title[ Evolution of the radius of analyticity]{ Evolution of the radius of analyticity \\ for mKdV-type equations}
\author[R. O. Figueira and M. Panthee]{Renata O. Figueira and  Mahendra  Panthee}
\address{\newline
Renata O. Figueira
\newline Department of Mathematics, University of Campinas (UNICAMP)
\newline 13083-859 Campinas, SP, Brazil
\newline email: {\tt re.oliveira.mat@gmail.com} 
\vspace{0.5cm}
\newline
Mahendra Panthee  \newline
Department of Mathematics,
 University of Campinas (UNICAMP)\newline
13083-859, Campinas, SP, Brazil
\newline
 email: {\tt mpanthee@unicamp.br}. \vspace{.5 cm}
}

\thanks{This work was partially supported by FAPESP, Brazil.}

\maketitle

\begin{abstract}
In this paper, we obtain new lower bounds for the evolution of the radius of  analyticity of solutions to two initial value problems (IVPs)  with initial data belonging to the class of analytic functions $H^{\sigma,s}(\rr)$ defined via a hyperbolic cosine weight. First, we consider the IVP for the modified Korteweg-de Vries (mKdV) equation.
For this problem, we prove that the evolution of the radius of  analyticity $\sigma(T)$ of the solution admits an algebraic lower bound $cT^{-\frac 12}$ for some $c>0$ and given arbitrarily large $T>0$. 
Next, we analyze the IVP for the mKdV equation with generalized dispersion (mKdVm) and a damping term. For this problem, we guarantee the local well-posedness in $H^{\sigma,s}(\rr)$ and demonstrate that the local solution can be extended globally in time and admits constant lower bounds for the radius of analyticity $\sigma(t)$ as time goes to infinity. 
The outcome of this paper concerning the mKdV equation represents an improvement on that achieved by the authors' previous work in \cite{FP24}. As far as we know, the results for the mKdVm with damping are new.
\end{abstract}
\vskip 0.3cm

 
 \vspace{0.2cm}

%
%
%
%
%
%

\section{Introduction and results}

In this work, we consider the initial value problems (IVPs) for the modified Korteweg-de Vries (mKdV) equation
\begin{equation}\label{mKdV-IVP}
\left\{\begin{array}{l}
\partial_t u+ \partial_x^{3}u+\mu u^2\p_x u
=
0,
\quad x,t\in\rr, \\
u(x,0)
=
u_0(x),
\end{array}\right.
\end{equation}
where $u$ is a real-valued function and $\mu=\pm1$,
and for the dispersion  generalized mKdV (mKdVm) equation with damping 
\begin{equation}\label{mKdVm-IVP}
\left\{\begin{array}{l}
\partial_t v+(-1)^{j+1} \partial_x^{m}v+\mu v^2\p_x v +a(x)v
=
0,
\quad x,t\in\rr, \\
v(x,0)
=
v_0(x),
\end{array}\right.
\end{equation}
where $m=2j+1$ with $j\in \{2,3,\dots\}$, and $v$ and $a$ are real-valued functions.

The mKdV equation \cite{M} emerges as an extension of the renowned Korteweg--de Vries (KdV) equation \cite{KdV}, distinguished by its focusing case for $\mu = 1$ and the defocusing case for $\mu = -1$.
Widely prevalent across diverse physical scenarios, the mKdV equation finds application in various domains. Notable examples include its role in plasma wave propagation \cite{D}, the dynamics of traffic flow \cite{W}, studies in fluid mechanics \cite{J} and investigations into nonlinear optics \cite{HK}.

Continuing the research started in \cite{BFH}, \cite{FP23} and \cite{FP24}, we consider the IVPs \eqref{mKdV-IVP} and \eqref{mKdVm-IVP} with given data in the Gevrey spaces  $G^{\sigma,s}(\mathbb{R})$. 
For $\sigma>0$ and $s\in \mathbb{R}$, the Gevrey space $G^{\sigma,s}(\mathbb{R})$ is generally defined as
$$
G^{\sigma,s}(\mathbb{R})
:=
\left\{ f\in L^2(\mathbb{R}):\;
\|f\|_{G^{\sigma,s}}
=
\Big(\int \la\xi\ra^{2s}e^{2\sigma |\xi|}|\widehat{f}(\xi)|^2 d{\xi}\Big)^\frac 12
<
\infty\right\},
$$
where $\la\xi\ra= (1+|\xi|)$ and $\widehat{f}$ denotes the Fourier transform given by
$$
\widehat{f}(\xi)
=\mathcal{F}(f)(\xi)
=
c\int e^{-ix\xi}f(x)d x.
$$
This space measures the regularity of the initial data and is, in some sense, a generalization of the classical $L^2$-based Sobolev space because for $\sigma=0$, $G^{0, s}(\R)$ simply becomes $H^s(\R)$.

Derived from the well-known Paley-Wiener theorems, it is established that a function $f\in G^{\sigma,s}(\mathbb{R})$ is analytic and possesses a holomorphic extension $\tilde{f}$ defined on the strip $S_\sigma=\{x+iy: |y|<\sigma\}$, satisfying $\sup_{|y|<\sigma} \|\tilde{f}(\cdot+iy)\|_{H^s}<\infty$ (we refer to page 174 in \cite{Ka} for details).
The parameter $\sigma>0$ is commonly referred to as the uniform radius of analyticity, as it determines the width of the strip where the function in $G^{\sigma,s}(\mathbb{R})$ can be holomorphically extended.

The well-posedness issues of the IVP \eqref{mKdV-IVP} with given data in the classical Sobolev spaces $H^s(\R)$ are well understood. For the sharp local well-posedness result, we refer to \cite{KPV93}, where the authors utilized the smoothing effect of the associated linear group combined with Strichartz and maximal function estimates to accomplish the work. This result was later reproduced in \cite{Tao} using trilinear estimates in the Bourgain space framework, which is defined as the completion of the Schwartz space with respect to the norm
\begin{equation}\label{Xsb}
\|w\|_{X^{s,b}}
=
\big\|\la\xi\ra^{s}\la\tau-\xi^{3}\ra^{b}\widehat{w}(\xi,\tau)\big\|_{L^2_{\xi}L^2_{\tau}},
\end{equation}
for $s,b\in\rr$. The Bourgain space $X^{s,b}(\rr^2)$, whose norm \eqref{Xsb} is defined using the restriction of the Fourier transform, has proved to be very convenient for dealing with the well-posedness issues for dispersive equations with low regularity  data.  Using this framework, the sharp global well-posedness result for the IVP \eqref{mKdV-IVP} with initial data in the classical Sobolev spaces $H^s(\R)$ was obtained in \cite{CKSTT}, see also \cite{CP-12}.

For the case of the IVP \eqref{mKdVm-IVP} with no damping, i.e., when $a = 0$, the local well-posedness result in $H^s(\rr)$ was shown in \cite{BFH}, where trilinear estimates in Bourgain spaces adapted to the higher-order dispersive term were derived.

Concerning the IVPs \eqref{mKdV-IVP} and \eqref{mKdVm-IVP} with initial data $u_0$ and $v_0$ in the Gevrey spaces $G^{\sigma_0,s}(\mathbb{R})$ for a fixed $\sigma_0>0$, the following natural questions arise: 

\begin{itemize}
\item Is it possible to obtain a local well-posedness result and maintain the radius of analyticity $\sigma_0$ during the local time of existence?
\item Is it possible to extend the local solutions $u(t)$ and $v(t)$ to any time interval $[0, T]$, and what is the behavior of the radius of analyticity, say $\sigma(t)$, as time progresses? If $\sigma(t)$ decreases as time evolves, is it possible to find a lower bound?
\end{itemize}

In recent times, these sorts of questions are widely studied for dispersive equations. For example, \cite{BFH, BGK, DMT, FP23, FP24, GTB, GK-1, SS, ST, T, TTB} are a few notable works.

For the IVP \eqref{mKdV-IVP} with initial data in $G^{\sigma, s}(\R)$, the first local well-posedness result was obtained in \cite{GK-1} whenever $s>\frac32$. The local well-posedness result proved in \cite{GK-1} ensures the existence of a finite time $T_0>0$, such that $\sigma(t)=\sigma_0$ for all $|t|\leq T_0$, indicating that the radius of analyticity of the solution remains the same as that of  the initial data for at least a small time interval.
The global well-posedness issue was addressed in \cite{BGK}, where the authors  proved that the local solution can be extended globally in time to any given interval $[0, T]$ and also obtained an algebraic lower bound $cT^{-12}$ for the evolution of the radius of analyticity. The next question one may naturally ask is whether the lower bound obtained in \cite{BGK} can be improved, in the sense that the decay rate of the radius of analyticity could be slower than the one obtained in \cite{BGK}, allowing the solution $u(t)$ to remain analytic in a larger strip for a longer time.

Motivated by the question raised in the previous paragraph, the authors in \cite{FP23} considered the IVP  \eqref{mKdV-IVP} with initial data in $G^{\sigma, s}(\R)$ and proved the local well-posedness result for $s\geq \frac14$ and also obtained $cT^{-\frac43}$ as an algebraic lower bound for the evolution of the radius of analyticity while extending the local solution globally in time, offering a significant improvement on the lower bounds for the radius of analyticity obtained in \cite{BGK}. This lower bound was further improved to $cT^{-1}$ in the authors' subsequent work \cite{FP24}.

In the case of the IVP \eqref{mKdVm-IVP} without damping (i.e., $a = 0$), local well-posedness in the Gevrey spaces $G^{\sigma, s}(\R)$ was established in \cite{BFH} under the condition
\begin{equation}
\label{s.index}
 s\ge s_m=\max\Big\{ -\frac{m-2}{6}, -\frac{14m-55}{60}\Big\}.
\end{equation} 
Furthermore, the authors demonstrated that the local solution could be extended globally in time, with a lower bound for the evolution of the radius of analyticity of order $T^{-\frac{1}{s_m}}$ for $m = 5, 7$, and $T^{-1}$ for $m \geq 9$.

The primary objective of this paper is to further improve  the lower bound for the evolution of the radius of analyticity of the solution of \eqref{mKdV-IVP}  obtained in \cite{FP24} and to leverage the damping term in \eqref{mKdVm-IVP} to establish a constant lower bound in this scenario, thereby improving the result obtained in \cite{BFH}. In the sequel, we provide an overview of the approach employed in this work.

The technique used in \cite{FP23} and \cite{FP24}, which is also employed here in this paper, essentially consists in proving an almost conserved quantity of the form
\begin{equation}
\label{ACL}
\sup\limits_{t\in[0,T_0]}
A_\sigma(t)
\le
A_\sigma(0) +C\sigma^\theta A_\sigma(0)^n, 
\;\; \theta>0,
\end{equation}
for all $0<\sigma\le \sigma_0$, where $C$ is a positive constant, $n\in\mathbb{N}$ and $A_\sigma(t)$ is associated with the conserved quantity satisfied by the flow of \eqref{mKdV-IVP} and, consequently, with the norm of the local solution at time $t\in [0, T_0]$, i.e., $\|u(\cdot,t)\|_{G^{\sigma,s}}$. This idea of almost conserved quantity was first introduced in \cite{ST} and has since been utilized by several authors, see, for example, \cite{SS}, \cite{FP23}, \cite{FP24} and references therein. With the almost conserved quantity of the form \eqref{ACL} at hand,  it is possible to extend the local solution $u$ globally in time  by decomposing any given time interval $[0,\,T]$ into short subintervals and iterating the local result in each subinterval. During this iteration process, a restriction emerges that provides a lower bound for the radius of analyticity $\sigma(T)$ as described below
$$
u\in C([-T,T]; G^{\sigma(T),s}(\rr)), \;\;\text{ with }\;\; \sigma(T)\ge cT^{-\frac 1\theta},
$$
where $\theta$ corresponds precisely to the value appearing in \eqref{ACL} and $c>0$ is a positive constant independent of $T$.
It is noteworthy that a larger value of $\theta$ yields a better lower bound for $\sigma(t)$, signifying a slower decay of the radius of analyticity as time goes to infinity.

In our previous work \cite{FP23}, the following almost conserved quantity at level $G^{\sigma, 1}(\mathbb{R})$ in the defocusing case ($\mu = -1$) was proved
\begin{equation}
\label{ACL1}
\sup\limits_{t\in[0,T_0]}
A_{I,\sigma}(t)
\leq
A_{I,\sigma}(0) + C\sigma^{\theta} A_{I,\sigma}(0)^2 \big(1+A_{I,\sigma}(0)\big),
\qquad \theta\in\Big[0, \frac34\Big],
\end{equation}
where $A_{I,\sigma}(t)$ is defined by
\begin{equation*}
A_{I,\sigma}(t)
=
\|e^{\sigma |D_x|}u(t)\|^2_{L_x^2}+\|\p_x(e^{\sigma |D_x|}u(t))\|^2_{L_x^2} -\frac{\mu}{6} \|e^{\sigma |D_x|}u(t)\|_{L^4_x}^4,
\end{equation*}
thus obtaining $cT^{-\frac43}$ as a lower bound for the radius of analyticity.
In the subsequent work \cite{FP24}, also in the defocusing case, an improvement on the interval to which the exponent $\theta$ belongs was achieved by proving the following almost conserved quantity at the level $G^{\sigma,2}(\mathbb{R})$
\begin{equation}
\label{ACL2}
\sup\limits_{t\in[0,T_0]}
A_{II,\sigma}(t)
\leq
A_{II,\sigma}(0) + C\sigma^\theta A_{II,\sigma}(0)^2 \big(1+A_{II,\sigma}(0)+A_{II,\sigma}(0)^2\big),
\qquad \theta\in [0,1],
\end{equation}
where $A_{II,\sigma}(t)$ is given by
\begin{equation}
\label{AII}
\begin{split}
A_{II,\sigma}(t)
&=
\|e^{\sigma |D_x|}u(t)\|^2_{L_x^2}+\|\p_x(e^{\sigma |D_x|}u(t))\|^2_{L_x^2} +\|\p_x^2(e^{\sigma |D_x|}u(t))\|^2_{L_x^2}\\
& \quad-
\frac{\mu}{6}\|e^{\sigma |D_x|}u(t)\|_{L^4_x}^4
-
\frac{5\mu}{3}
\|(e^{\sigma |D_x|}u)\cdot\p_x(e^{\sigma |D_x|}u)(t)\|_{L^2_x}^2
+
\frac{1}{18}
\|e^{\sigma |D_x|}u(t)\|_{L_x^6}^6,
\end{split}
\end{equation}
and, consequently, obtaining $cT^{-1}$ as a lower bound for the evolution of the radius of analyticity of the solution.

Since the mKdV equation has infinitely many conserved quantities, the natural path to follow would be to consider the almost conserved quantities at the higher levels of regularity.
However, a crucial observation must be made: both proofs of the inequalities \eqref{ACL1} and \eqref{ACL2} rely on a particular property of the exponential function
\begin{equation}
\label{exp-est}
e^{\sigma |\xi|}-1
\leq
(\sigma |\xi|)^\theta e^{\sigma |\xi|},
\quad \theta\in[0,1].
\end{equation}
The presence of the exponential weight  $e^{\sigma |\xi|}$ in the definition of the Gevrey space  $G^{\sigma,s}(\rr)$ naturally leads to the use of an estimate of the form \eqref{exp-est}. Therefore, due to the necessity of inequality \eqref{exp-est} in the technique presented in the proofs of \eqref{ACL1} and \eqref{ACL2}, the highest achievable value for $\theta$ is $\theta=1$, resulting in $cT^{-1}$ as the best lower bound for the radius of analyticity.  This is the reason for considering the almost conserved quantities only up to level $G^{\sigma,2}(\rr)$ using this approach with the exponential weight in the norm of $G^{\sigma,s}(\rr)$. With this observation in mind, to get a better lower bound for the evolution of the radius of analyticity one needs to find some other alternative to avoid the use of the estimate \eqref{exp-est} in its current form.

\subsection{Function spaces and the statement of the main results}
Motivated by the recent contributions in \cite{DMT}, \cite{GTB} and \cite{TTB}, our current investigation considers a modified Gevrey norm by replacing the exponential weight $e^{\sigma |\xi|}$ in $\| \cdot\|_{G^{\sigma,s}}$ with a hyperbolic cosine weight $\cosh(\sigma |\xi|)=\cosh(\sigma \xi)$, defined as follows
\begin{equation}
\| f\|_{H^{\sigma,s}}
:=
\Big(\int \la\xi\ra^{2s}\cosh^2(\sigma \xi)|\widehat{f}(\xi)|^2 d{\xi}\Big)^\frac 12.
\end{equation}

The norms $\| \cdot\|_{G^{\sigma,s}}$ and $\| \cdot\|_{H^{\sigma,s}}$ are equivalent due to the trivial estimate
\begin{equation}
\label{cosh-exp-est}
\frac 12 e^{\sigma |\xi|} 
\le
\cosh (\sigma \xi)
\le
e^{\sigma |\xi|}.
\end{equation}
This equivalence highlights that both norms control the same function space, with only a slight difference in the way the analytic growth is captured.
Therefore, the operator $\cosh(\sigma D_x)$, defined by 
$$
\mathcal{F}(\cosh(\sigma D_x)\varphi)(\xi)
=
\cosh(\sigma \xi)\widehat{\varphi}(\xi),
$$
serves as a bridge between the classical Sobolev space $H^s(\rr)$ and the class of analytic functions $H^{\sigma,s}(\rr)$. Specifically, it satisfies the relation 
$$\|f\|_{H^{\sigma,s}}=\|\cosh(\sigma D_x)f\|_{H^{s}}, $$
highlighting that the analytic growth introduced by this operator characterizes a function in $H^{\sigma,s}(\mathbb{R})$ as one whose image under $\cosh(\sigma D_x)$ belongs to $H^s(\mathbb{R})$.

The use of the hyperbolic cosine  weight in the Gevrey norm offers an advantage as the function $\cosh(\sigma \xi)$ complies with the estimate
\begin{equation}\label{cosh-1}
\cosh(\sigma\xi)-1\leq (\sigma|\xi|)^{2\theta}\cosh(\sigma\xi), \qquad \theta\in [0,1],
\end{equation}
which appears while constructing almost conserved quantities, see \eqref{Fourier-F} below.
Observe that, considering $\theta =1$ in \eqref{cosh-1}, it is possible  to achieve an exponent $2$ for $\sigma$ in the almost conserved quantity, which translates to $cT^{-\frac 12}$ as the new algebraic lower bound for the radius of analyticity (see Theorems~\ref{ACL-mKdV-thm} and \ref{global-mKdV-thm}  below).

Considering hyperbolic cosine as the weight function in the space measuring the regularity of the initial data, it is reasonable to extend this consideration to the Bourgain spaces.
We consider the space $Y_m^{\sigma,s,b}(\rr^2)$ given by the norm
\begin{equation}
\label{Ysb-norm}
\|w\|_{Y_m^{\sigma,s,b}}
:=
\big\| \cosh(\sigma\xi)\la\xi\ra^{s}\la\tau-\xi^{m}\ra^{b}\widehat{w}(\xi,\tau)\big\|_{L^2_{\xi,\tau}},
\quad m\in\{3,5,7,\ldots\}.
 \end{equation}

Also, for $T>0$ we denote the Gevrey-Bourgain space restricted in time to $(-T,T)$ by $Y_{m,T}^{\sigma,s,b}(\rr^2)$ with norm given by
\begin{equation*}
\| w\|_{Y_{m,T}^{\sigma,s,b}}
:=
\inf\big\{ \|\tilde{w}\|_{Y_m^{\sigma,s,b}}:\;
w=\tilde{w}
\text{ on } \rr\times (-T,T)\big\}.
\end{equation*} 
When $m=3$ we simplify the notation by writing $Y^{\sigma,s,b}:= Y_3^{\sigma,s,b}$ and similarly $Y_{T}^{\sigma, s,b}:= Y_{3,T}^{\sigma, s,b}$. 
Additionally, if $\sigma=0$, we omit the analytic parameter, denoting $ Y^{s,b}_m:=Y^{0,s,b}_m$.
In the special case when $\sigma=0$ and $m=3$, the norm reduces to the classical Bourgain space norm given by \eqref{Xsb}, meaning that $Y_3^{0,s,b}\equiv Y^{s,b}\equiv X^{s,b}$. 

Again, in light of \eqref{cosh-exp-est}, one observes that $\|\cdot\|_{X_m^{\sigma,s,b}}$ and  $\|\cdot\|_{Y_m^{\sigma,s,b}}$ are equivalent, where $\|\cdot\|_{X_m^{\sigma,s,b}}$ is considered in \cite{BFH} and \cite{FP23} is given by 
$$
\|w\|_{X_m^{\sigma,s,b}}
:=
\big\| e^{\sigma|\xi|}\la\xi\ra^{s}\la\tau-\xi^{m}\ra^{b}\widehat{w}(\xi,\tau)\big\|_{L^2_{\xi,\tau}}.
$$

Taking into consideration that the norms   $\| \cdot\|_{G^{\sigma,s}}$ and $\| \cdot\|_{H^{\sigma,s}}$, as well as  $\|\cdot\|_{X^{\sigma,s,b}}$ and  $\|\cdot\|_{Y^{\sigma,s,b}}$, are equivalent, with the procedure used in \cite{FP23}, one can easily prove the following local well-posedness result.

\begin{theorem}\label{lwp-mKdV-thm}
Let $\sigma_0>0$, $b>\frac12$ and $s\ge \frac14$. For each $u_0\in H^{\sigma_0,s}(\rr)$ there exists a time 
\begin{equation}\label{lifetime}
T_0=T_0(\|u_0\|_{H^{\sigma_0,s}})=\frac{c_0}{(1+\|u_0\|_{H^{\sigma_0,s}}^2)^d},\qquad c_0>0, \quad d>1,
\end{equation}
 such that the IVP \eqref{mKdV-IVP} admits a unique solution $u$ in $ C([-T_0,T_0] ; H^{\sigma_0,s}(\rr))\cap Y_{T_0}^{\sigma_0,s,b}$ satisfying
\begin{equation}\label{bound.u}
\sup\limits_{|t|\le T_0}\|u(t)\|_{H^{\sigma_0,s}}
\le
\|u\|_{Y^{\sigma_0,s,b}_{T_0}} 
\le 
C\|u_0\|_{H^{\sigma_0,s}}.
\end{equation}
\end{theorem}

Taking into account the presence of a damping term in the mKdVm equation, we first establish the local well-posedness result for the IVP \eqref{mKdVm-IVP} in $H^{\sigma,0}(\rr)$ as outlined below in Theorem \ref{lwp-mKdVm}. 
Before proceeding further, the function $a$, referred to as the damping term, is initially assumed to satisfy the following conditions.

\medno {\bf (A1) Damping effect.} There exists $\lambda>0$ such that
\begin{equation*}
a(x)\ge \lambda, 
\;\; \text{ for all } x\in \rr.
\end{equation*}
\medno {\bf (A2)  Regularity.} There exist positive constants $C$ and $R$ such that
\begin{equation}
\label{A2}
\|\p_x^ka\|_{L^{\infty}}
\le
C R^k k!, \;\;\text{ for all }k\in\mathbb{N}.
\end{equation}
Additionally, in order to control the constant that appears in \eqref{A2}, the next condition is also required.
\medno {\bf (A3) Controlling the norm.} 
The constant  $R>0$ appearing in \eqref{A2} satisfies
$$
R<\frac 1\sigma_0,
$$
where $\sigma_0$ is the radius of analyticity of the initial data.

Moreover, we introduce the function space $\mathcal{A}^\sigma$, where $\sigma > 0$, defined as the completion of the space of analytic functions, under the norm 
\begin{equation}
\label{A-norm}
\|a\|_{\mathcal{A}^{\sigma}}
=
\sum\limits_{k=0}^\infty (k+1)^\frac 14 \frac{\sigma^k}{k!} \|\p_x^ka\|_{L^\infty}.
\end{equation}
This norm is motivated by the result provided in Lemma \ref{bilinear.au.thm} below.
It is important to point out that the condition {\bf (A3)} implies $\|a\|_{\mathcal{A}^{\sigma_0}}<\infty$, which means $a\in \mathcal{A}^{\sigma_0}$.

We are now prepared to present the local well-posedness result  for the IVP \eqref{mKdVm-IVP} with initial data in \(H^{\sigma,0}(\mathbb{R})\).

\begin{theorem}
\label{lwp-mKdVm}
Let $m\ge 5$, $\sigma_0\ge 0$ and the damping  $a$ satisfies the conditions {\bf (A1)}-{\bf (A3)}.
For each $v_0\in H^{\sigma_0,0}(\rr)$ there exist $b> 1/2$ and a time
$$
T_0
=
\frac{c_0}{(1+ \|a\|_{\mathcal{A}^{\sigma_0}}+\|v_{0}\|^2_{H^{\sigma_0,0}})^d}, \qquad c_0>0, \quad d>1,
$$
such that the IVP \eqref{mKdVm-IVP} admits a unique solution
$v\in C([-T_0,T_0]; H^{\sigma_0,0}(\rr))\cap Y^{\sigma_0,0,b}_{m,T_0} $
satisfying
\begin{equation}
\label{sol.bound}
\sup\limits_{|t|\le T_0}\|v(t)\|_{H^{\sigma_0,0}}
\le
c\|v\|_{Y^{\sigma_0,0,b}_{m,T_0}}
\le
\|v_{0}\|_{H^{\sigma_0,0}}.
\end{equation}
Furthermore, the data-to-solution map is continuous.
\end{theorem}

In order to extend the local solutions obtained in Theorems \ref{lwp-mKdV-thm} and \ref{lwp-mKdVm} globally in time and obtain the lower bound for the evolution of the radius of analyticity, as mentioned above, we will  derive almost conserved quantities in the $H^{\sigma, 2}(\rr)$ and $H^{\sigma,0}(\rr)$ spaces, respectively. 
For this, we consider the following quantities in the modified Gevrey class $H^{\sigma,s}(\rr)$ 
\begin{equation}
\label{A}
\begin{split}
A_\sigma(t)
&:=
\|\cosh(\sigma D_x)u(t)\|^2_{L_x^2}+\|\p_x[\cosh(\sigma D_x)u(t)]\|^2_{L_x^2} +\|\p_x^2[\cosh(\sigma D_x)u(t)]\|^2_{L_x^2}\\
& \quad-
\frac{\mu}{6}\|\cosh(\sigma D_x)u(t)\|_{L^4_x}^4
-
\frac{5\mu}{3}
\|[\cosh(\sigma D_x)u]\cdot\p_x[\cosh(\sigma D_x)u](t)\|_{L^2_x}^2\\
&\quad +
\frac{1}{18}
\|\cosh(\sigma D_x)u(t)\|_{L_x^6}^6,
\end{split}
\end{equation}
which comes from \eqref{AII} by replacing $e^{\sigma |D_x|}$ by $\cosh(\sigma D_x)$, and
\begin{equation}
\label{M}
M_\sigma(t):=\|\cosh(\sigma D_x)v(t)\|_{L^2_x}^2,
\end{equation}
which, for $\sigma =0$, is the $L^2$-conserved quantity for the IVP \eqref{mKdVm-IVP}.
 
With this framework at hand, in the following theorems we will derive almost conserved quantities that play a crucial role in proving the main results of this work. We start by stating the almost conserved quantity associated with the IVP \eqref{mKdV-IVP}.
\begin{theorem}
\label{ACL-mKdV-thm}
Let $\sigma >0$, $u_0\in H^{\sigma,2}(\rr)$ and $u\in C([-T_0,T_0]; H^{\sigma,2}(\rr))$ be the local solution of the IVP \eqref{mKdV-IVP} given by Theorem \ref{lwp-mKdV-thm}.
Let $A_\sigma(t)$ be as defined in \eqref{A}. Then for some $b>1/2$ and  for all $t\in [0,T_0]$, we have
\begin{equation}
\label{ACL-mKdV}
A_\sigma(t)
\le
A_\sigma(0) + C\sigma^2 \|u\|^4_{Y^{\sigma,2,b}_{T_0}} \big(1+\|u\|^2_{Y^{\sigma,2,b}_{T_0}}+\|u\|^4_{Y^{\sigma,2,b}_{T_0}}\big).
\end{equation}
Moreover, considering $\mu=-1$, one has
\begin{equation}
\label{ACL-mKdV-0}
A_\sigma(t)
\le
A_\sigma(0) + C\sigma^2 A_\sigma(0)^2 \big(1+A_\sigma(0)+A_\sigma(0)^2\big).
\end{equation}
\end{theorem}

To derive an almost conserved quantity  associated with the IVP \eqref{mKdVm-IVP}, an additional condition  $a\in H^{\sigma_0,0}(\rr)$ on the damping function $a$ is required.  The necessity of this extra condition on the damping term is explained in Remark \ref{a-remark} below. Now we state the almost conserved quantity for the IVP \eqref{mKdVm-IVP}.

\begin{theorem}
\label{acq-mKdVm}
Let $m\ge 5$, $0<\sigma\le \sigma_0$, $v_0\in H^{\sigma,0}(\rr)$, $a\in H^{\sigma_0,0}(\rr)$ satisfying {\bf (A1)}-{\bf (A3)} and $v\in C([-T_0,T_0];H^{\sigma,0}(\rr))$ be the local solution to the IVP \eqref{mKdVm-IVP} guaranteed by  Theorem \ref{lwp-mKdVm}. Then, for all $t\in [0,T_0]$, we have
\begin{equation}
\label{acq-mKdVm.eq}
M_\sigma(t)
\le
e^{-2\lambda t}M_\sigma(0)
+ 
C_1\big(\sigma^\theta M_\sigma(0) +\sigma \|a\|_{\mathcal{A}^{\sigma_0}}\big)M_\sigma(0),
\end{equation}
for any $\theta\in[0,\min\{1,- s_m\}]$  (see \eqref{s.index} for the definition of $s_m$).
\end{theorem}


\begin{remark}\label{w-remark}
Concerning the almost conserved quantity \eqref{acq-mKdVm.eq}, the exponential decay $e^{-2\lambda t}$ clearly demonstrates the dissipative effect induced by the damping term in \eqref{mKdVm-IVP}. This phenomenon was previously studied by Wang in \cite{Wa}, where a damping term was introduced in the KdV equation. The author in \cite{Wa} constructed an almost conserved quantity in $G^{\sigma, 0}$-space  of the form 
\begin{equation}
\label{acq-W}
\|u(\delta)\|^2_{G^{\sigma,0}}
\le
e^{-2\lambda\delta}\|u(0)\|^2_{G^{\sigma,0}}
+
C_1\sigma^\alpha\|u(0)\|^3_{G^{\sigma,0}}
+ C_2\|a\|_{L^{\infty}}\|u_0\|_{L^2}\|u_0\|_{G^{\sigma, 0}}, \qquad \alpha>0,
\end{equation}
where $\delta>0$ is the lifespan of the local solution $u$ of the IVP for the KdV equation with damping $a$ satisfying the conditions {\bf (A1)}-{\bf (A3)}. The last term with $C_2$ in \eqref{acq-W} is referred to as a ``negligible term" in \cite{Wa}. However, in the iteration process to extend the local solution globally in time using \eqref{acq-W}, this negligible term accumulates. As such  the local existence time in each step of iteration  tends to shrink and one may not  reach the arbitrarily large given time $T>0$. This happens because there is no $\sigma$ in the negligible term which would, somehow, facilitate the control of the growth of the norms involved. 

In our study, we considered $H^{\sigma,0}(\rr)$-space and obtained a simpler almost conserved quantity \eqref{acq-mKdVm.eq} for the mKdVm equation with damping, without any ``negligible terms''. This simplification is  applicable to the KdV equation as well. The refined formulation clarifies the proof and facilitates a more streamlined approach to extending solutions globally in time.

It is worth noting that the additional hypothesis $a\in H^{\sigma_0,0}(\rr)$ appears to be necessary in Wang’s framework as well, as discussed in Remark \ref{a-remark} below. While Wang’s almost conserved quantity is valid under the condition $a\in H^{\sigma_0,0}(\rr)$, the proof of the global extension of solutions given therein encounters challenges due to the presence of ``negligible terms''. By employing the techniques developed in our study, one can simplify the proof of the global extension for the KdV equation with damping. This alternative approach not only verifies the main result of Wang’s work but also demonstrates its validity in $H^{\sigma,0}(\rr)$-space with a more direct and effective analysis.

Furthermore, it is important to clarify that the analyticity assumption on the damping coefficient $a(x)$ is a technical necessity within the framework of analytic Bourgain spaces. Since our analysis relies on the control of the solution using the $\cosh(\sigma D_x)$ weight, the interaction term $a(x)u$ must preserve this exponential decay on the Fourier side. In the Fourier domain, this interaction becomes a convolution $\widehat{a} * \widehat{u}$. If $a(x)$ lacked analyticity, this convolution could potentially lead to a loss of regularity in the analytic sense, preventing the boundedness of the estimates for the almost conserved quantity. Therefore,  analyticity of the damping is a strong requirement in our approach to ensure that the damping effect does not introduce singularities that would cause the radius of analyticity to collapse.

We believe that these improvements contribute to a deeper understanding of the dissipative mechanisms induced by damping and provide a robust framework for further studies in this direction.
\end{remark}

As observed in \cite{FP24} (see Remark $4.1$ there), there is a price to pay when we use the $\cosh(\sigma \xi)$ instead of $e^{\sigma |\xi|}$ as the weight function. With the hyperbolic cosine, one needs to absorb two derivatives  while constructing the almost conserved quantity \eqref{ACL-mKdV}. 
The discovery of how to handle these extra derivatives can be found in the proof of Theorem \ref{ACL-mKdV-thm},  particularly during the derivation of inequality \eqref{est-r4}. 
The key ingredients to obtain \eqref{est-r4} are the following  estimates satisfied by the linear group associated with the linear part of the IVP \eqref{mKdV-IVP} 
\begin{equation}
\label{kpv-1}
\| e^{-t\p_x^3} u_0\|_{L_x^2L^\infty_T}
\le
C\|u_0\|_{H^s}, \text{ for all } T\lesssim 1, \; u_0\in H^s(\rr) \;\text{ and }\; s> \frac 34,
\end{equation}
whose proof can be found in \cite{KPV91-1} (see Corollary 2.9 there)  and
\begin{equation}
\label{kpv-2}
\|\p_x e^{-t\p_x^3} u_0\|_{L_x^\infty L^2_T}
\le
C\|u_0\|_{L^2}, \text{ for all } u_0\in L^2(\rr),
\end{equation}
whose proof can be found in \cite{KPV93}, see also \cite{KPV91}.

Finally, we are in a position to state the main results of this work that deal with the  evolution in time of the radius of spatial analyticity of the local solutions when extended globally.  

The first main result is concerned with a new lower bound for the radius of analyticity of the solution to the IVP \eqref{mKdV-IVP} and reads as follows.
\begin{theorem}\label{global-mKdV-thm}
Let $\sigma_0>0$, $s\ge 1/4$, $u_0\in H^{\sigma_0,s}(\rr)$ and $u\in C([-T_0,T_0];H^{\sigma_0,s}(\rr))$ be the local solution of the IVP \eqref{mKdV-IVP} guaranteed by Theorem \ref{lwp-mKdV-thm}. Then, in the defocusing case $(\mu = -1)$, for any $T\ge T_0$, the solution $u$ can be extended globally in time and
$$
u\in C\big([-T,T]; H^{\sigma(T),s}(\rr)\big),
\quad\text{with}\quad
\sigma(T)\ge\min\Big\{\sigma_0, cT^{-\frac 12}\Big\},
$$
with $c$ being a positive constant depending on $s$, $\sigma_0$ and $\|u_0\|_{H^{\sigma_0, s}}$. 
\end{theorem}

\begin{remark}
Regarding the lower bound $\sigma(T) \ge cT^{-1/2}$ established in Theorem \ref{global-mKdV-thm}, it is worth mentioning that, within the framework of an almost conserved quantity and the $\cosh$ weight used in this work, this decay rate is the best possible. This limitation arises directly from the factor  $\sigma^2$ obtained in the estimates for the error term (see Lemma \ref{est.fU-mKdV}). Any further improvement to this exponent would likely require more sophisticated approaches, such as the application of the Inverse Scattering Transform for the integrable case, which fall outside the scope of the energy method developed in this paper.

Observe that for the solution to the IVP \eqref{mKdV-IVP} in the focusing case ($\mu = 1$), the quantity $A_\sigma(t)$ defined in \eqref{A} involves terms with negative signs. Consequently, one cannot directly ensure that $\|u_0\|_{H^{\sigma,2}}^2 \le A_\sigma(t)$ (see \eqref{defocusing} below), which is a necessary step to obtain a bound of the form \eqref{ACL-mKdV-0}. Since the estimate \eqref{ACL-mKdV-0} plays a crucial role in the global iteration argument in the proof of Theorem~\ref{global-mKdV-thm}, we considered the defocusing mKdV equation in order to establish the lower bound for the evolution of the radius of analyticity. This situation was also observed in the previous works \cite{FP23} and \cite{FP24}.

\end{remark}

The second main result deals with the constant lower bound for the evolution of the radius of analyticity of the solution to the IVP \eqref{mKdVm-IVP} which is obtained exploiting  the term $e^{-2\lambda t}<1$ in \eqref{acq-mKdVm.eq}.

\begin{theorem}\label{global-mKdVm}
Let $m\ge 5$, $\sigma_0>0$, $v_0\in H^{\sigma_0,0}(\rr)$, and the damping $a\in H^{\sigma_0,0}(\rr)$ satisfies the conditions {\bf (A1)}-{\bf (A3)}.  If $v\in C([-T_0,T_0];H^{\sigma_0,0}(\rr))$ is the local solution of the IVP \eqref{mKdVm-IVP} guaranteed by Theorem \ref{lwp-mKdVm}, then there is a constant $\tilde{\sigma}_0=\tilde{\sigma}_0(\|v_0\|_{H^{\sigma_0,0}}, \sigma_0, a)$ such that the solution $v$ can be extended globally in time 
and
\begin{equation}
\label{const.radius}
v(t) \in H^{\sigma(t),0}(\rr),
\quad\text{with}\quad
\sigma(t) \ge \tilde{\sigma}_0 >0, \;\;\text{ for all } t\ge 0.
\end{equation}
Additionally, the following bound holds
\begin{equation}
\label{exp.decay}
\|v(t)\|_{H^{\tilde{\sigma}_0,0}}
\le
Ce^{-\frac{\lambda t}{2}}, \;\;\text{ for all } t\ge 0,
\end{equation}
with $C$ being a positive constant depending on $\|v_0\|_{H^{\sigma_0,0}}$. 
\end{theorem}

\noindent
{\bf Notations:} Throughout the text we will use standard notations that are commonly used in partial differential equations.  Besides $\widehat{f}$, the notation $\mathcal{F}_xf$ is also used to denote the partial Fourier transform with respect to the spatial variable $x$.  We use $c$ or $C$ to denote various  constants whose exact values are immaterial and may
 vary from one line to the next. We use $A\lesssim B$ to denote an estimate
of the form $A\leq cB$ and $A\sim B$ if $A\leq cB$ and $B\leq cA$.

The structure of this paper unfolds as follows. Section \ref{Sec-2} provides fundamental preliminary estimates. Section \ref{Sec-3} is devoted to the proof of Theorem \ref{lwp-mKdVm}. In Section \ref{Sec-4}, the almost conserved quantities are established as stated in Theorems \ref{ACL-mKdV-thm} and \ref{acq-mKdVm}. Finally, in Section \ref{Sec-5} the lower bounds for the radius of analyticity are obtained by proving Theorems \ref{global-mKdV-thm} and \ref{global-mKdVm}.

%
%
%
%
%
%
\section{Preliminaries} \label{Sec-2}

In this section, we record the technical tools required for our analysis. Throughout this work, we assume the standard properties of the Bourgain spaces $Y_m^{\sigma,s,b}$, such as the continuous inclusion 
$Y_m^{\sigma,s,b} \subset C([-T,T], H^{\sigma,s}(\rr))$ for $b > 1/2$, and the following time-localization estimate
\begin{equation}
\label{inc.Tao}
\|\chi_I w\|_{Y^{\sigma,s,\beta}_m} \lesssim \|w\|_{Y^{\sigma,s,\beta}_{m,T}}
\end{equation}
which holds for $|\beta| < 1/2$ and any interval $I \subset [-T, T]$, where $\chi_I$ denotes the characteristic function of $I$. 
For a comprehensive treatment of these spaces and their embedding properties, we refer to \cite{Tao-b, BFH1, SS}.

We continue with the statement of the trilinear estimates for the mKdVm equation. 
A similar proof can be found in \cite{BFH} (see Corollary $1$ there).

\begin{lemma}
\label{trilinear.mKdVm}
Let $m\in\{5,7,9,\ldots\}$, $\sigma\ge0$ and $s\ge s_m$, where $s_m$ is given as in \eqref{s.index}. There exist $\frac 12<b<1$ and $-\frac 12< b'<0$ such that the following estimate holds
\begin{equation}
\label{trilinear.mKdVm.eq}
\|\p_x(v_1 v_2 v_3)\|_{Y_m^{\sigma,s,b'}}
\lesssim
\|v_1\|_{Y_m^{\sigma,s,b}} \|v_2\|_{Y_m^{\sigma,s,b}}\|v_3\|_{Y_m^{\sigma,s,b}},
\end{equation}
for any $\sigma>0$ and $v_1,v_2,v_3\in Y_m^{\sigma,s,b}$.
\end{lemma}

The proof of the next result follows closely the steps presented in the proof of Lemma $2.3$ in \cite{Wa}.

\begin{lemma}
\label{bilinear.au.thm}
Let  $\sigma \ge 0$, $m\in\{3,5,7,\ldots\}$, and $b'\le 0\le b$. Then, the following inequality holds
\begin{equation}
\label{bilinear.au}
\|av\|_{Y_m^{\sigma,0,b'}}
\lesssim
\|a\|_{\mathcal{A}^\sigma}\|v\|_{Y_m^{\sigma,0,b}},
\end{equation} 
for all $v\in Y_m^{\sigma,0,b}$ and $a\in\mathcal{A}^\sigma$.
\end{lemma}

\begin{remark}
It should be emphasized that Lemma \ref{bilinear.au.thm} is the reason we restrict the local well-posedness result for \eqref{mKdVm-IVP} to the space $Y_m^{\sigma,s,b}$ with $s=0$. While it is possible to prove a similar inequality to \eqref{bilinear.au} for any $s> 1/2$, using the equivalence
$$
\|f\|_{H^{\sigma,s}}
\equiv
\sum\limits_{k=0}^\infty
\Big((k+1)^\frac 12 \Big(\frac{\sigma^k}{k!}\Big)^2\|\p_x^k f\|_{H^{0,s}}^2\Big)^\frac 12,
$$
and the fact that the Sobolev space $H^s(\rr):=H^{0,s}(\rr)$ is closed under multiplication in this case, the argument for $s\in (0,  1/2]$ remains unresolved. We focus on $s=0$ as it is the most significant index in the context of the almost conserved quantity described in Theorem \ref{acq-mKdVm}, which is a central objective of this paper.
\end{remark}

As mentioned in the introduction, when applying the hyperbolic cosine weight in proving the almost conservation law stated in Theorem \ref{ACL-mKdV-thm}, two extra derivatives emerge. The subsequent lemma presents crucial inequalities for managing these additional derivatives.
Throughout, we denote by $|D^p_x|$ the differential operator with symbol $|\xi|^p$ and by  $a+$ the quantity $a+\varepsilon$ for any $\varepsilon>0$. 

\begin{lemma}
\label{strichartz.combined}
Let $b > 1/2$. Then the following estimates hold.
\begin{enumerate}
    \item For $p \in \{6, 8\}$,
    \begin{equation} \label{str-1}
    \|u\|_{L^p_xL^p_t} \lesssim \|u\|_{Y^{0,b}}. 
    \end{equation}
    \item For any $U_1, U_2, U_3, U_4 \in Y^{0,b}$, we have
    \begin{equation}
    \label{strichartz.lemma}
        \|U_1 U_2 U_3\|_{L^{2}_xL_{t}^2} \lesssim \prod_{j=1}^3 \|U_j\|_{Y^{0,b}} \quad \text{and} \quad \|U_1 U_2 U_3 U_4\|_{L^{2}_xL_{t}^2} \lesssim \prod_{j=1}^4 \|U_j\|_{Y^{0,b}}.
    \end{equation}
    \item For $0 < T \lesssim 1$, 
    \begin{equation}
    \label{U-L2}
    \| U\|_{L_x^2L_T^\infty} \lesssim \|U\|_{Y^{\frac{3}{4}+,b}_T} \quad \text{and} \quad \| |D_x| U\|_{L_x^\infty L_T^2} \lesssim \|U\|_{Y^{0,b}_T}.
    \end{equation}
\end{enumerate}
\end{lemma}

\begin{proof}
The proof of \eqref{str-1} for  $p=8$ follows from the famous Strichartz estimate for the Airy group obtained in \cite{KPV91} (see also  \cite{Axel-1}). Interpolating this with the trivial estimate $\|w\|_{L^2_{x}L^2_t}\lesssim \|w\|_{Y^{0,0}}$ yields \eqref{str-1} for  $p=6$, see \cite{LG}.
Inequalities in \eqref{strichartz.lemma} are an immediate consequence of the  Strichartz estimate \eqref{str-1} and the generalized H\"older inequality. 

Using classical transfer arguments (see, for example,  Lemma $2.9$ in \cite{Tao-b}), the inequalities in \eqref{U-L2} follow directly from \eqref{kpv-1} and \eqref{kpv-2}, respectively. 
\end{proof}

\begin{remark}
The restriction $0 < T \lesssim 1$ in \eqref{U-L2} is consistent with our iterative argument, which is performed over bounded time intervals of fixed width.
\end{remark}

The next lemma provides the primary motivation for adopting the weight  $\cosh(\sigma \xi)$ instead of $e^{\sigma |\xi|}$ to measure the analyticity of functions in this work. The proof of this result will rely on the following inequality
\begin{equation}
\label{sinh-ineq}
|\sinh(r)|\le |r|^\theta \cosh r, 
\; \text{for all } r\in\rr \text{ and }\theta\in [0,1].
\end{equation}
The validity of \eqref{sinh-ineq} is established in the following cases.

\medno
{\bf Case I: $\theta\in[0,1]$ and $|r|\ge1$ or $\theta=0$ and $r\in\rr$.} In this case, \eqref{sinh-ineq} is trivial.

\medno
{\bf Case II: $\theta=1$ and $r\in\rr$.} Since $\cosh$ is increasing on $[0,|r|]$, we have
$$
|\sinh(r)|
=
|\sinh(|r|)|
=
\Big|\int_0^{|r|}\cosh(s) ds \Big|
\le
\cosh(|r|)\int_0^{|r|} ds
=
|r|\cosh r.
$$

\medno
{\bf Case III: $\theta\in [0,1]$ and $|r|<1$.} First, we can assume $r\neq 0$, since for the case $r=0$ the desired inequality \eqref{sinh-ineq} is trivial. We define $f(\theta)=|r|^\theta \cosh(r)-|\sinh(r)|$ and observe that
$$
f'(\theta)=\ln(|r|) |r|^\theta \cosh(r) < 0, \;\;\text{ for all } \theta\in (0,1],
$$
indicating that $f$ is decreasing on the interval $(0,1]$. From cases {\bf I} and {\bf II} we know that $f(0), f(1)\ge 0$, thus proving $f(\theta)\ge 0$, for all $\theta\in [0,1]$.

\begin{lemma}
\label{cosh-est}
Let $\xi = \xi_1+\xi_2+\xi_3$, $\xi_k\in\rr$ for $k=1,2,3$, and $\sigma>0$.
Then, for all $\theta_1,\theta_2\in [0,1]$ we have
\begin{equation}\label{e2.8}
\Big|1- \cosh(\sigma \xi) \sech(\sigma \xi_1)\sech(\sigma \xi_2)\sech(\sigma \xi_3) \Big|
\lesssim
\sigma^{\theta_1 +\theta_2} \xi_{\text{med}}^{\theta_1} \xi_{\max}^{\theta_2}, 
\end{equation}
where $\xi_{\text{med}}$ and $\xi_{\text{max}}$ denote the median and maximum values of $\{ |\xi_1|,|\xi_2|, |\xi_3|\}$, respectively.
\end{lemma}

\begin{proof}
Recalling the fundamental equalities
\begin{equation*}
\cosh(x+y) 
=
\cosh(x)\cosh(y) +\sinh(x)\sinh(y)
\quad\text{and}\quad
\sinh(x+y) 
=
\sinh(x)\cosh(y) +\cosh(x)\sinh(y),
\end{equation*}
and using the inequality \eqref{sinh-ineq}, we can write
\begin{equation*}
|\cosh(\sigma\xi_1)\cosh(\sigma\xi_2)\cosh(\sigma\xi_3) -\cosh(\sigma\xi)|
\lesssim
\sigma^{\theta_1 +\theta_2}\xi_{\text{med}}^{\theta_1}\xi_{\max}^{\theta_2}
\cosh(\sigma\xi_1)\cosh(\sigma\xi_2)\cosh(\sigma\xi_3). 
\end{equation*}
The proof of \eqref{e2.8} follows  by multiplying both sides of the last inequality by $\sech(\sigma\xi_1)\sech(\sigma\xi_2)\sech(\sigma\xi_3)$.
\end{proof}

The following lemmas address the estimates for the key terms that arise in the proof of the almost conserved quantities within the analytic function spaces $H^{\sigma,2}(\rr)$ for the mKdV equation  \eqref{mKdV-IVP} and $H^{\sigma,0}(\rr)$ for the mKdVm equation  \eqref{mKdVm-IVP}. 
Specifically, these lemmas provide bounds for the norm of the operators $F$ and $G$, defined as follows
\begin{equation}
\begin{split}
\label{fU-def-mKdV}
F(W)
&:=
\frac {\mu}{3} \partial_x\Big[
W^3-\cosh(\sigma D_x)\big((\sech(\sigma D_x) W)^3\big)
\Big]\\
G(W)
&:=
a(x)W-\cosh(\sigma D_x)\big[a(x)(\sech(\sigma D_x) W)\big],
\end{split}
\end{equation}
where $\mu=\pm 1$ and $a$ is the damping function appearing in \eqref{mKdVm-IVP}.

\begin{lemma}
\label{est.fU-mKdV}
Let $\sigma>0$ and $F$ be defined as in \eqref{fU-def-mKdV}.
Then, there is some $\frac12<b<1$ such that
\begin{align}
\label{F-L2}
\|F(U)\|_{L^2_xL^2_t}
&\lesssim
\sigma^{2}\|U\|^3_{Y^{2,b}},\\
\label{UF-L2}
\|U F(U)\|_{L^2_xL^2_t}
&\lesssim
 \sigma^{2}\|U\|^4_{Y^{2,b}}.
\end{align}
\end{lemma}
\begin{proof}
We begin by writing the following estimate for the Fourier transform of $F$, a result that can be readily obtained using Lemma \ref{cosh-est} with $\theta_1=\theta_2=1$,
\begin{equation}
\label{Fourier-F}
\!\!
|\widehat{F(U)}(\xi,\tau)|
\lesssim
|\xi|\!\int_\ast \big|1-\cosh(\sigma \xi) \prod\limits_{k=1}^3 \sech(\sigma \xi_k) \big|
\prod_{j=1}^3 |\widehat{U}(\xi_j,\tau_j)|
\lesssim
\sigma^2 |\xi| \!
\int_\ast \xi_{\text{med}}\xi_{\max}
\prod_{j=1}^3 |\widehat{U}(\xi_j,\tau_j)|
\end{equation}
where $\int_\ast$ denotes the integral over the set 
$\{(\xi, \tau)\in \rr^2:  \xi=\xi_1+\xi_2+\xi_3\text{ and }\tau=\tau_1+\tau_2+\tau_3\}$.
Without loss of generality, by symmetry, we can assume $|\xi_1|\le |\xi_2| \le |\xi_3|$, and using inequality \eqref{Fourier-F}, we derive
\begin{equation}
\label{F-est}
|\widehat{F(U)}(\xi,\tau)|
\lesssim
\sigma^2 \int_\ast |\xi_2||\xi_3|^2
|\widehat{U}(\xi_1,\tau_1)\widehat{U}(\xi_2,\tau_2)\widehat{U}(\xi_3,\tau_3)|.
\end{equation}

To establish \eqref{F-L2}, we utilize Parseval's identity and apply estimate \eqref{F-est} to obtain
$$
\|F(U)\|_{L^2_xL^2_t}
\lesssim
\sigma^2 \| W_0 W_1 W_2\|_{L^2_x L^2_t},
$$
where 
\begin{equation}
\label{w-def}
\widehat{W_\kappa}=\la\xi\ra^\kappa|\widehat{U}|, \qquad \kappa= 0,1,2.
\end{equation} 
Then, it follows from \eqref{strichartz.lemma} that for all $b> 1/2$ 
\begin{equation}
\label{F-L2-X1}
\|F(U)\|_{L^2_x L^2_t}
\lesssim
 \sigma^2 \|W_0\|_{Y^{0,b}}\|W_1\|_{Y^{0,b}}\|W_2\|_{Y^{0,b}}
\lesssim
\sigma^2 \|U\|_{Y^{0,b}}\|U\|_{Y^{1,b}}\|U\|_{Y^{2,b}},
\end{equation}
thus ensuring the validity of inequality \eqref{F-L2}.

Regarding \eqref{UF-L2}, we proceed similarly. By applying estimate \eqref{F-est} and Parseval's identity, we derive
\begin{equation*}
\|U F(U)\|_{L^2_x L^2_t} \lesssim \sigma^2 \|W_0^2 W_1W_2\|_{L^2_x L^2_t},
\end{equation*}
where $W_\kappa$ is as defined in \eqref{w-def}. The result follows immediately from the quadrilinear Strichartz estimate in \eqref{strichartz.lemma}, yielding $\|U F(U)\|_{L^2_x L^2_t} \lesssim \sigma^2 \|U\|^2_{Y^{0,b}} \|U\|_{Y^{1,b}} \|U\|_{Y^{2,b}}$.
\end{proof}

\begin{lemma}
\label{Xsb-F.mKdVm}
Let $m\in\{5,7,9,\ldots\}$, $\sigma >0$, $\theta\in[0,\min\{1,-s_m\}]$, where $s_m$ is given by \eqref{s.index}, and let $F$ be defined as in \eqref{fU-def-mKdV}.
Then, there exist $-1/2 <b'<0$ and $1/2<b\le 1+b'$ such that
\begin{equation*}
\big\|F(V)\big\|_{Y^{0,b'}_{m}}
\lesssim
\sigma^\theta \|V\|^3_{Y^{0,b}_{m}}
\end{equation*}
for all $V\in Y^{0,b}_{m}$.
\end{lemma}

\begin{proof}
By applying the same arguments used to derive \eqref{F-est}, but now setting $\theta_1=\theta$ and $\theta_2=0$, we obtain
\begin{equation}
\label{F-est2}
|\widehat{F(V)}(\xi,\tau)|
\lesssim
\sigma^\theta |\xi|\int_\ast |\xi_2|^\theta
\prod_{j=1}^3 |\widehat{V}(\xi_j,\tau_j)|,
\end{equation}
where we considered $|\xi_1|\le |\xi_2| \le |\xi_3|$. Then, we have
\begin{equation*}
\|F(V)\|_{Y^{0,b'}_m} \lesssim \sigma^\theta \Big\| \la \tau-\xi^m\ra^{b'} \la\xi\ra^{-\theta} |\xi| \int_\ast \la\xi_2\ra^\theta \la\xi_3\ra^\theta \prod_{j=1}^3 |\widehat{V}(\xi_j,\tau_j)| \Big\|_{L^2_{\xi, \tau}} \sim \sigma^\theta \|\p_x(w_0w_\theta^2)\|_{Y^{-\theta,b'}_m},
\end{equation*}
where $w_\kappa$ is defined by $\widehat{w_\kappa}=\la\xi\ra^\kappa |\widehat{V}|, \; \kappa= 0, \theta$.
The result then follows by applying the trilinear estimate \eqref{trilinear.mKdVm.eq} with $s=-\theta$ valid for $\theta\in [0, \min\{1,-s_m\}]$, which yields $\|F(V)\|_{Y^{0,b'}_m} \lesssim \sigma^\theta \|V\|_{Y^{0,b}_m}^3$. This finishes the proof.
\end{proof}

\begin{lemma}
\label{aU-est}
Let $a\in \mathcal{A}^{\sigma_0}\cap H^{\sigma_0,0}(\rr)$ with $\sigma_0>0$. If $0<\sigma\le \sigma_0$, then
\begin{equation}
\label{H-est}
\|G(V)\|_{L^2_x}
\lesssim
\frac{\sigma}{\sigma_0} \|a\|_{\mathcal{A}^{\sigma_0}}\|V\|_{L^2_x},
\end{equation} 
for all $V\in Y^{0,b}_{m}(\rr^2)$ with $b\ge 0$, where $G$ is given by \eqref{fU-def-mKdV}.
\end{lemma}

\begin{proof}
We define the operator $e^{\pm \sigma D_x}$ given by $\mathcal{F}(e^{\pm \sigma D_x}\varphi)=e^{\pm \sigma\xi}\widehat{\varphi}$. Using Fourier transform, one can easily see that
$
e^{\pm \sigma D_x}(av)=(e^{\pm \sigma D_x} a)\cdot(e^{\pm \sigma D_x}v),
$
implying
$$
G(V)
=
-\frac 12\Big[
(e^{\sigma D_x}-1)a\cdot e^{\sigma D_x}v 
+
(e^{-\sigma D_x}-1)a\cdot e^{-\sigma D_x}v 
\Big],
$$
where $v=\sech(\sigma D_x) V$.

In this way, it follows from the Cauchy-Schwarz inequality that
\begin{equation}
\label{GV-L2.2}
\| G(V)\|_{L^2_x}
\le
\|(e^{\sigma D_x} -1)a\|_{L^\infty}
\| V\|_{L^2_x}
+
\|(e^{-\sigma D_x} -1)a\|_{L^\infty}
\| V\|_{L^2_x},
\end{equation}
where the Parseval identity and inequality \eqref{cosh-exp-est} were used to ensure that $\frac{1}{2}\|e^{\pm \sigma D_x}v\|_{L^2_x}\le \|V\|_{L^2_x}$.

Since $a\in H^{\sigma_0,0}(\rr)$, $a$ admits a holomorphic extension $\tilde{a}$ defined on $S_{\sigma_0} =\{ x+iy: |y|<\sigma_0\}$, as mentioned in the introduction. Therefore, we can write
\begin{equation}
\label{correct-exp.a}
[(e^{\sigma D_x} -1)a](x)
=
\tilde{a}(x-i\sigma)-\tilde{a}(x), 
\;\;\text{ for all }x\in \rr \text{ and } 0<\sigma< \sigma_0,
\end{equation}
as can be easily verified by applying the Fourier transform to both sides of the equality.
Expanding $\tilde{a}$ in Taylor series, we have
\begin{equation}
\label{exp.a1}
\begin{split}
\|(e^{\sigma D_x} -1)a\|_{L^\infty}
=
\| \tilde{a}(x-i\sigma)-\tilde{a}(x)\|_{L^\infty}
\le \frac{\sigma}{\sigma_0} \sum\limits_{k=1}^\infty \frac{\sigma_0^k}{k!} \|\p_x^k a\|_{L^\infty}
=
\frac{\sigma}{\sigma_0} \| a\|_{\mathcal{A}^{\sigma_0}},
\end{split}
\end{equation}
since $0<\sigma\le \sigma_0 $.
In a similar way, we can write
\begin{equation}
\label{exp.a2}
\|(e^{-\sigma D_x} -1)a\|_{L^\infty}
\le
\frac{\sigma}{\sigma_0} \| a\|_{\mathcal{A}^{\sigma_0}}.
\end{equation}

Therefore, the proof is finished by combining \eqref{GV-L2.2} with \eqref{exp.a1} and \eqref{exp.a2}.
\end{proof}

\begin{remark}\label{a-remark}

We note that the proof of Lemma \ref{aU-est} follows closely the steps of the proof of Lemma 3.3 presented in \cite{Wa}. The main difference in our approach lies in the additional assumption that the damping function $a$ belongs to the analytic function space $H^{\sigma_0,0}(\rr)$. This extra assumption, which slightly weakens our result, is necessary due to the equality \eqref{correct-exp.a} used in our proof.
However, we must respectfully point out a discrepancy in \cite{Wa}, where the author uses the following expression 
\begin{equation*}  
\big((e^{\sigma D}-1)a\big)(x)
=
a(x+\sigma)-a(x) 
\end{equation*} 
in which  the imaginary unit $i$ is missing in the argument of the first function of the RHS (see equation $(3.27)$ there). The correct form should be $a(x+i\sigma)$, but this does not make sense because $a$ is defined on $\R$.  To overcome this problem we considered  $a \in H^{\sigma_0}(\mathbb{R})$.  In our view, this additional assumption is crucial for the application of the pseudo-differential operator $e^{\sigma D_x}$ to $a$. Specifically, the well-definedness of $e^{\sigma D_x}a$ requires that the Fourier transform of $a$ exhibits exponential decay, which is guaranteed under this assumption. It would be nice, if one can overcome this problem with less restrictive condition on the damping term $a$.

It is worth noticing that this constraint does not arise when $e^{\sigma D_x}$ is applied to the product $av$, as the exponential decay of $\mathcal{F}_x(av)$ is established in Lemma \ref{bilinear.au.thm}. Thus, the local well-posedness result remains valid under the weaker regularity assumptions on the damping function $a$ as stated in \cite{Wa}.
\end{remark}

To conclude this section, we present an additional technical lemma, whose proof closely follows the approach used in the proof of Lemma $3.1$ in \cite{Wa}.

\begin{lemma}
\label{Wa-3.1}
For all $m\in\{5,7,9, \ldots\}$, $1/2 <b<1$ and $T>0$, we have
$$
\|\chi_{[0,T]}e^{-2\lambda(T-t)}V\|_{Y_m^{0,1-b}}
\lesssim
\|V\|_{Y_{m,T}^{0,b}}.
$$
\end{lemma}

%
%
%
%
%
%
\section{Local well-posedness of  the IVP \eqref{mKdVm-IVP} - Proof of Theorem \ref{lwp-mKdVm} }
\label{Sec-3}
This section provides a detailed proof of Theorem \ref{lwp-mKdVm}. 
First, using Duhamel formula, we observe that a formal solution to the IVP \eqref{mKdVm-IVP} is given by
$$
v(x,t)
=
W(t)v_0(x)-\int_0^tW(t-t')(\mu v^2\p_xv +av)dt',
$$
where $W(t)$ is the semigroup defined by its Fourier multiplier as follows
$$\mathcal{F}_x(W(t)\varphi)=e^{i\xi^m t}\widehat{\varphi}(\xi).$$
By using a cutoff function $\psi\in C_0^\infty(-2,2)$ with $0\le \psi\le 1$ and $\psi(t)=1$ for all $t\in[-1,1]$, we define the following map
\begin{equation}
\label{Phi}
\Phi_Tv(x,t)
=
\psi(t)W(t)v_0(x)-\psi_T(t)\int_0^tW(t-t')(\mu v^2\p_xv +av)dt',
\end{equation}
for any $0< T\le 1$ with $\psi_T(t)=\psi(t/T)$.

With this setting, proving the existence of a solution to the IVP \eqref{mKdVm-IVP} on $[-T,T]\times \rr$ reduces to finding a fixed point of  $\Phi_T$ in a suitable complete metric space. In other words, the proof of Theorem~\ref{lwp-mKdVm}  consists in demonstrating the existence of a $T>0$ such that $\Phi_T$ is a contraction map in a closed ball $B(r)=\{v\in Y_{m,T}^{\sigma_0,0,b}; \|v\|_{Y^{\sigma_0,0,b}_{m,T}}\le r\}$, for some $r>0$. The next proposition goes in this direction and its proof largely replicates the arguments in \cite{BFH} (see Lemma $2$ there).

\begin{proposition}
\label{map.est}
Let $m\in \{5,7,9\ldots\}$, $\sigma_0>0$, $-1/2<b'\le 0\le b\le 1-b'$ and $0<T\le 1$. For any $v_0\in H^{\sigma_0,0}(\rr)$, we have
\begin{equation}
\| \Phi_T(v)\|_{Y_m^{\sigma_0,0,b}}
\lesssim
 \|v_0\|_{H^{\sigma_0,0}}+T^{1-(b-b')}\|\mu v^2\p_xv +av\|_{Y_{m}^{\sigma_0,0,b'}},
\end{equation}
where $\Phi_T$ is given by \eqref{Phi}.
\end{proposition}

\begin{proof}[Proof of Theorem \ref{lwp-mKdVm}]
Let us fix $\sigma_0> 0$ and $v_0\in H^{\sigma_0,0}(\rr)$, and consider the map $\Phi_T$ given in \eqref{Phi}.
The combination of Proposition \ref{map.est}, Lemmas \ref{trilinear.mKdVm} and \ref{bilinear.au.thm} leads to the following estimate for the Bourgain norm of $\Phi_T$
\begin{equation*}
\|\Phi_T(v)\|_{Y^{\sigma_0,0,b}_m}
\lesssim
\|v_0\|_{H^{\sigma_0,0}} + T^{d'}(\|a\|_{\mathcal{A}^{\sigma_0}}+\|v\|_{Y_m^{\sigma_0,0,b}}^2)\|v\|_{Y_m^{\sigma_0,0,b}},
\end{equation*}
and
\begin{equation*}
\begin{split}
\|\Phi_T(v_1)-\Phi_T(v_2)\|_{Y_m^{\sigma_0,0,b}}
&\lesssim
T^{d'}\|\mu (v_1^2\p_x v_1-v_2^2\p_xv_2)+a(v_1-v_2)\|_{Y_m^{\sigma_0,0,b'}}\\
&\sim
T^{d'}\|\mu \p_x[(v_1^2 +v_2^2+ v_1v_2 )(v_1-v_2)]+a(v_1-v_2)\|_{Y_m^{\sigma_0,0,b'}}\\
&\lesssim
T^{d'}\!(\|v_1\|_{Y_m^{\sigma_0,0,b}}^2
\!+\!
\|v_2\|_{Y_m^{\sigma_0,0,b}}^2
\!+\!
\|v_1\|_{Y_m^{\sigma_0,0,b}}\|v_2\|_{Y_m^{\sigma_0,0,b}}
\!+\!
\|a\|_{\mathcal{A}^{\sigma_0}})\|v_1-v_2\|_{Y_m^{\sigma_0,0,b}}
\end{split}
\end{equation*}
for some $b>\frac 12$ and $0<d'< 1$.
Therefore, by choosing
$$
r\sim \|v_0\|_{H^{\sigma_0,0}} 
\quad \text{and} \quad
T\sim \frac{1}{(1+\|a\|_{\mathcal{A}^{\sigma_0}}+\|v_0\|^2_{H^{\sigma_0,0}})^\frac {1}{d'}},
$$
we can show that $\Phi_T$ maps the ball $B(r)$ into itself and is a contraction mapping within this set. Consequently, the existence of a solution $v$ in the space $Y_{m,T}^{\sigma_0,0,b}(\rr^2)$ is established for some $b>1/2$.
Moreover, by the continuous inclusion $Y_m^{\sigma,s,b} \subset C([-T,T], H^{\sigma,s}(\rr))$ for $b > 1/2$ and the appropriate choice of $r$, this local solution satisfies the bound given in \eqref{sol.bound}.

The uniqueness in the entire space $Y_{m,T}^{\sigma_0,0,b}(\rr^2)$ and the continuity of the solution map $\Phi_T$ are derived using standard techniques. For a complete exposition of the arguments, we refer the reader to \cite{BFH1}, \cite{BFH}, and the references therein.
\end{proof}

%
%
%
%
%
%
\section{Almost conserved quantities - Proof of Theorems \ref{ACL-mKdV-thm} and \ref{acq-mKdVm} } \label{Sec-4}

This section is dedicated to providing a detailed proof of the almost conserved quantities described in Theorems \ref{ACL-mKdV-thm} and \ref{acq-mKdVm}.
These results are essential in the proofs of Theorems \ref{global-mKdV-thm} and \ref{global-mKdVm}, serving as fundamental components of the arguments presented therein.

\begin{proof}[Proof of Theorem \ref{ACL-mKdV-thm}] We start proving \eqref{ACL-mKdV}. Note that the expression for $A_{\sigma}(t)$ provided in \eqref{A} can be represented as
\begin{equation}\label{A-m1}
 A_\sigma(t)=\mathcal{I}_{0}(t) +\mathcal{I}_{1}(t)+\mathcal{I}_{2}(t),
 \end{equation}
 where $U:=\cosh(\sigma D_x)u$ and
\begin{equation}\label{I_012}
\begin{split}
\mathcal{I}_{0}(t)
&=
\int [U(x,t)]^2dx,\\
\mathcal{I}_{1}(t)
&=
\int [\p_x U(x,t)]^2 dx-\frac{\mu}{6}\int [U(x,t)]^4 dx,\\
\mathcal{I}_{2}(t)
&=
\int [\p_x^2U(x,t)]^2dx +\frac{1}{18}\int [U(x,t)]^6dx -\frac{5\mu}{3}\int [U(x,t)\p_xU(x,t)]^2dx.
\end{split}
\end{equation}

Moreover, by applying the operator $\cosh(\sigma D_x)$ to the mKdV equation \eqref{mKdV-IVP}, we obtain
\begin{equation}
\label{mKdV-U}
\p_tU+ \p_x^3 U +\mu U^2\p_x U
=
F(U),
\end{equation}
where the operator $F$ is given by \eqref{fU-def-mKdV}.

Now, by following the same steps presented in the proof of Proposition~$3.1$ in \cite{FP24} (more precisely, see equation $(3.21)$ there), it can easily be shown that
\begin{equation}
\label{td-A}
\begin{split}
\frac{d}{dt}A_\sigma(t) 
&= 
2\int UF(U)dx +2\int \p_xU\p_x[F(U)]dx -\frac {2\mu}{3} \int U^3F(U) dx+2\int \p^2_x U \p_x^2 [F(U)]dx\\
& 
\quad 
+
\frac13\int U^5F(U)dx
+\frac{10\mu}{3}\int U(\p_xU)^2F(U)dx
+\frac{10\mu}{3}\int U^2\p_x^2UF(U)dx.
\end{split}
\end{equation}

Integrating \eqref{td-A} over the time interval $[0,t]$ with $0<t<T$, we obtain
\begin{equation}
\label{goal.A}
A_\sigma(t)
=
A_\sigma(0) + R(t),
\end{equation}
where $R(t):= R_{1}(t)+ R_{2}(t)+R_{3}(t)+R_4(t)$ with
\begin{align*}
R_{1}(t)
&:=
2\iint \chi_{[0,t]}UF(U)dxdt', \\
R_{2}(t)
&:=
-2\iint \chi_{[0,t]}\p_x^2 U F(U)dxdt'
-\frac{2\mu}{3}\iint \chi_{[0,t]} U^3F(U)dxdt',\\
R_{3}(t)
&:=
\frac13\iint  \chi_{[0,t]}U^5F(U)dxdt' 
+\frac{10\mu}{3}\iint  \chi_{[0,t]}[U(\p_xU)^2F(U)
+U^2\p_x^2UF(U)]dxdt',\\
R_{4}(t)
&:=
2\iint  \chi_{[0,t]}\p^2_x U \p_x^2 [F(U)]dx dt'.
\end{align*}

In view of \eqref{goal.A}, the estimate \eqref{ACL-mKdV} will follow if we can show that
\begin{equation}
\label{goal.r}
|R(t)|
\lesssim
\sigma^2 \|u\|^4_{Y^{\sigma,2,b}_T} \big(1+\|u\|^2_{Y^{\sigma,2,b}_T}+\|u\|^4_{Y^{\sigma,2,b}_T}\big),
\end{equation}
for all $t\in [0,T]$.

By applying duality, the time-restriction estimates \eqref{inc.Tao} and \eqref{strichartz.lemma} and Lemma \ref{est.fU-mKdV}, we bound the first three remainder terms as follows
\begin{equation}\label{est-r-short}
|R_{1}(t)| + |R_{2}(t)| + |R_{3}(t)| \lesssim \sigma^2 \| U\|^4_{Y_T^{2,b}} \big(1+\|U\|^2_{Y_T^{2,b}}+\|U\|^4_{Y_T^{2,b}}\big).
\end{equation}

In this way, to get  the desired inequality \eqref{goal.r}, the only remaining task is to estimate the absolute value of the term $R_4(t)$, thereby completing the proof of \eqref{ACL-mKdV}.
To achieve this, employing the same steps to derive \eqref{F-est}, we observe that the Fourier transform of $F(U)$ restricted to the spatial variable can be bounded as follows
\begin{equation}
\label{F-est-x}
|\mathcal{F}_x(F(U))(\xi,t')|
\lesssim
\sigma^2 \int_\bullet |\xi_2||\xi_3|^2
|\mathcal{F}_x(U)(\xi_1,t')\mathcal{F}_x(U)(\xi_2,t') \mathcal{F}_x(U)(\xi_3,t')|, 
\end{equation}
if we assume $|\xi_1|\le |\xi_2| \le |\xi_3|$, where $\int_\bullet$ denotes the integral over the set $\{\xi\in\rr: \;\xi=\xi_1+\xi_2+\xi_3\}$. 

Applying Parseval identity in the $x$ variable and estimate \eqref{F-est-x}, we get
\begin{equation*}
\begin{split}
|R_4(t)|
&\lesssim
\sigma^2
\int \chi_{[0,t]}\bigg(\int \Big[\overline{\mathcal{F}_x(|D_x|^3w)}(\xi,t')\Big]
\Big[\mathcal{F}_x( w\cdot |D_x|w \cdot |D_x^3| w)(\xi,t')\Big] d\xi\bigg) dt'
\\
&=
\sigma^2 
\iint \chi_{[0,t]} \overline{|D_x|^3w} \cdot w\cdot |D_x|w \cdot |D_x^3| w \; dxdt',
\end{split}
\end{equation*}
since $|\xi|\le 3|\xi_3|$, where $\mathcal{F}_xw=|\mathcal{F}_x(U)|$.

Now, it follows from Holder inequality and estimates in \eqref{U-L2} that
\begin{equation}
\label{est-r4}
|R_4(t)|
\lesssim
\sigma^2 
\||D_x|^3w\|_{L_x^\infty L^2_T}^2  \|w\|_{L_x^2 L^\infty_T} \||D_x|w\|_{L_x^2 L^\infty_T}
\lesssim
\sigma^2 \|U\|_{Y_{T}^{2,b}}^4.
\end{equation}

From the estimates \eqref{est-r-short} and \eqref{est-r4} it readily follows that $R(t)$ satisfies \eqref{goal.r}, 
which concludes the proof of \eqref{ACL-mKdV}.

In what follows, we will supply a  proof of \eqref{ACL-mKdV-0}. Using the bound \eqref{bound.u}  of the local solution in  the almost conserved quantity \eqref{ACL-mKdV}, we get 
\begin{equation}\label{est-At}
A_\sigma(t)
\le
A_\sigma(0) +C\sigma^2 \|u_0\|_{H^{\sigma,2}}^4(1+\|u_0\|_{H^{\sigma,2}}^2+\|u_0\|_{H^{\sigma,2}}^4).
\end{equation}
Hence, the proof of \eqref{ACL-mKdV-0} follows from \eqref{est-At} by noticing that 
\begin{equation}
\label{defocusing}
 \|u_0\|^2_{H^{\sigma,2}}
\le  
A_{\sigma}(0), \text{ when $\mu=-1$}.
\end{equation}
This finishes the proof of Theorem \ref{ACL-mKdV-thm}.
\end{proof}

Before providing a proof of Theorem \ref{acq-mKdVm}, we observe that if $v$ is a solution of \eqref{mKdVm-IVP}, then
\begin{equation}
\label{dt.L2-norm.v}
\frac{d}{dt}\int v(x,t)^2 dx
=
-2\int a(x)v^2(x,t)dx.
\end{equation}
Indeed, to derive \eqref{dt.L2-norm.v}, we write
\begin{equation}
\label{dt.L2-norm.v.1}
\frac{d}{dt}\int v^2 dx
=
2\int v\p_tv dx
=
-2(-1)^{j+1}\int v\p_x^m v  dx -2\mu\int v^3\p_x v dx - 2\int av^2 dx
\end{equation}
since $v$ satisfies the mKdVm equation. 
Applying integration by parts, we obtain that the first two integrals on the right-hand side of \eqref{dt.L2-norm.v.1} vanish due to the boundary conditions and the structure of the nonlinearity, leaving only the term involving $av^2$.

Therefore, it follows from \eqref{dt.L2-norm.v} and the Gronwall inequality that
\begin{equation}
\label{L2-norm.v-exp}
\|v(t)\|_{L_x^2}^2
\le 
e^{-2\lambda t}\|v_0\|^2_{L^2},
\end{equation}
for all $t\ge 0$.
Here,  the constant $\lambda$ is specified in condition {\bf (A1)} for the damping function.

At this point, the objective is to establish an upper bound on the growth of the solution in the analytic space $H^{\sigma,0}(\rr)$. This will be achieved through the proof presented below.

\begin{proof}[Proof of Theorem \ref{acq-mKdVm}] 
Let $V=\cosh(\sigma D_x)v$. By applying the operator $\cosh(\sigma D_x)$ to the mKdVm equation \eqref{mKdVm-IVP}, we obtain
\begin{equation*}
\p_tV+(-1)^{j+1}\p_x^m V +\mu V^2\p_x V +aV
=
F(V) +G(V),
\end{equation*}
where $F$ and $G$ are as defined in \eqref{fU-def-mKdV}.
Applying the same arguments to prove \eqref{dt.L2-norm.v}, we get
\begin{equation*}
\frac{d}{dt}\int V^2 dx
+
2\int aV^2 dx
=
2\int (F(V)+G(V))V dx. 
\end{equation*}
Since $a(x)\ge \lambda>0$ by condition $\bf{(A1)}$, we have
\begin{equation}\label{vv-1}
\frac{d}{dt}\|V(t)\|_{L_x^2}^2 
\le
-2\lambda \|V(t)\|_{L_x^2}^2 
+2\int (F+G)V dx. 
\end{equation}
For clarity and conciseness, we suppress the arguments of the operators $F$ and $G$ without compromising the understanding of the reasoning.
Now, using the classical differential form of Gronwall's lemma, it follows from \eqref{vv-1} that
\begin{equation}
\|V(t)\|_{L_x^2}^2 
\le
e^{-2\lambda t}\|V(0)\|_{L^2}^2 
+
2\int_0^t e^{-2\lambda(t-t')} \Bigg(\int (F+G)Vdx\Bigg)dt',
\end{equation}
for all $t\in [0, T_0]$.

Therefore, since $M_\sigma(t)= \|V(t)\|_{L_x^2}^2$, we have
\begin{equation}
\label{M-R}
M_\sigma(t)
\le
e^{-2\lambda t}M_\sigma(0)
+
2(R_{I}+R_{II}),
\end{equation}
where
\begin{equation*}
R_I
:=
\Bigg| \int_{0}^{t}\int e^{-2\lambda(t-t')}FVdx dt' \Bigg|
\quad
\text{and}
\quad
R_{II}
:=
\Bigg| \int_{0}^{t}\int e^{-2\lambda(t-t')}GVdx dt' \Bigg|.
\end{equation*}

Using duality, Lemmas \ref{Wa-3.1} and \ref{Xsb-F.mKdVm} restricted to time, and inequality \eqref{inc.Tao}, we can estimate $R_{I}$ as follows
\begin{equation}
\label{RI.est}
\begin{split}
R_{I}
\lesssim
\|\chi_{[0,t]} F\|_{Y_{m}^{0,b'}}\|\chi_{[0,t]}e^{-2\lambda(t-t')} V \|_{Y_{m}^{0,-b'}}
\lesssim
\|F\|_{Y_{m,T_0}^{0,b'}}\|V\|_{Y^{0,b}_{m,T_0}}
\lesssim
\sigma^\theta\|V\|^4_{Y^{0,b}_{m,T_0}}
\lesssim \sigma^\theta M^2_\sigma(0),
\end{split}
\end{equation}
where we used the fact that $-b'<1-b< 1/2<b$ and the solution bound \eqref{sol.bound}.

In order to estimate $R_{II}$, we apply Cauchy-Schwarz inequality in the $x$ variable and estimate \eqref{H-est} to obtain
\begin{equation}\label{R2-0}
R_{II}
\le
\int_0^{t} e^{-2\lambda(t-t')}\|G\|_{L^2_x}\|V\|_{L^2_x} dt' 
\lesssim
\frac{\sigma}{\sigma_0}\|a\|_{\mathcal{A}^{\sigma_0}}\int_0^{t} e^{-2\lambda(t-t')}\|v\|^2_{H^{\sigma,0}}dt'.
\end{equation}
Now, using the solution bound \eqref{sol.bound}, the estimate \eqref{R2-0} yields
\begin{equation}
\label{RII.est}
R_{II}
\lesssim
\frac{\sigma}{\sigma_0}\|a\|_{\mathcal{A}^{\sigma_0}}\|v_0\|^2_{H^{\sigma,0}}\frac{(1-e^{-2\lambda t})}{2\lambda}
\lesssim
\frac{\sigma}{\sigma_0}\|a\|_{\mathcal{A}^{\sigma_0}}\|v_0\|^2_{H^{\sigma,0}}\frac{(1-e^{-2\lambda T_0})}{2\lambda}
\lesssim
\sigma\|a\|_{\mathcal{A}^{\sigma_0}}M_\sigma(0),
\end{equation}
for all $t\in [0, T_0]$.
The proof of \eqref{acq-mKdVm.eq} is completed by combining \eqref{M-R}, \eqref{RI.est},  and \eqref{RII.est}.
\end{proof}

\begin{remark}
\label{acq-exp}
In inequality \eqref{acq-mKdVm.eq}, we note that the exponential $e^{-2\lambda t}$ can be written as $e^{-2\lambda (t - 0)}$, where $0$ is the lower bound of the time interval $[0, t] \subset [0, T_0]$, representing the length of the interval. This becomes crucial when applying \eqref{acq-mKdVm.eq} to the intervals $[kT_0, (k+1)T_0]$ for $k=0,1,\dots$, in the proof of Theorem \ref{global-mKdVm}, as each interval has the same length. Consequently, the exponential term remains consistent in each step, specifically $e^{-2\lambda T_0}$, simplifying the analysis across successive intervals.\end{remark}

%
%
%
%
%
%
\section{Lower bounds for the radius of analyticity - Proof of Theorems  \ref{global-mKdV-thm} \& \ref{global-mKdVm}}
\label{Sec-5}

In this section, we use the almost conserved quantities established in the previous section and provide proofs of the main results of this work. Having obtained the almost conserved quantities, the idea of the proof of Theorems  \ref{global-mKdV-thm} and \ref{global-mKdVm} is similar to those used  in \cite{FP23}, \cite{FP24} and \cite{Wa}.

\begin{proof}[Proof of Theorem \ref{global-mKdV-thm}]
The proof follows the iterative scheme presented in \cite{FP23}.
By fixing $s=2$ and using the $\sigma^2$ gain in the almost conserved quantity \eqref{ACL-mKdV-0}, we obtain the lower bound $\sigma(T)\geq cT^{-\frac12}$. 
\end{proof}

\begin{proof}[Proof of Theorem \ref{global-mKdVm}] 
Let $\sigma_0>0$ and $v_0\in H^{\sigma_0,0}(\rr)$.
The main idea in the proof is to iterate the local solution  by using the almost conserved quantity \eqref{acq-mKdVm.eq}.
Recall that the lifespan $T_0$ of the local solution $v$ to the IVP \eqref{mKdVm-IVP} guaranteed by Theorem \ref{lwp-mKdVm} is given by
$$
T_0
=
\frac{c_0}{(1+\|a\|_{\mathcal{A}^{\sigma_0}}+M_{\sigma_0}(0))^d},
$$
where $M_{\sigma}(t)$ is defined as in \eqref{M}.
The proof  is done by an iterative argument on the positive intervals $[0,T_0], [T_0, 2T_0],\dots$. 
This suffices because the mKdVm equation is time-reversible, allowing us to apply the argument in both temporal directions.

Let $0<\sigma\le\sigma_0$ to be determined later.
In Theorems \ref{lwp-mKdVm} and \ref{acq-mKdVm}, we established the existence of a solution $v\in C([0,T_0]; H^{\sigma,0}(\rr))$ that satisfies
\begin{equation}
\label{acq-mKdVm.eq.1}
M_\sigma(T_0)
\le
e^{-2\lambda T_0}M_\sigma(0)
+ 
C_1\big(\sigma^\theta M_\sigma(0) +\sigma \|a\|_{\mathcal{A}^{\sigma_0}}\big)M_\sigma(0),
\end{equation}
since $0<\sigma\le\sigma_0$.

For the next step,  we choose 
\begin{equation}
\label{cond.1}
\sigma:=\min\bigg\{\sigma_0, \frac{1-e^{-2\lambda T_0}}{2C_1\|a\|_{\mathcal{A}^{\sigma_0}}}, \Big(\frac{1-e^{-2\lambda T_0}}{2C_1M_{\sigma_0}(0)}\Big)^{1/\theta}\bigg\}.
\end{equation}

For the choice of $\sigma$ in \eqref{cond.1}, it follows from \eqref{acq-mKdVm.eq.1} that
\begin{equation}
\label{bound.1}
M_{\sigma}(T_0)\le M_{\sigma_0}(0)
\qquad \mathrm{and\; consequently}\qquad
\|v(T_0)\|_{H^{\sigma,0}}\le \|v_0\|_{H^{\sigma_0,0}}.
\end{equation}
The estimates \eqref{bound.1} under the  choice of $\sigma$ in \eqref{cond.1} allow us to apply Theorem \ref{lwp-mKdVm} with initial time $T_0$ (in place of $0$) and initial data $v(T_0)$ (in place of $v_0$), thereby obtaining a solution to the IVP \eqref{mKdVm-IVP} for the mKdVm equation at the interval  $[T_0, T_0+ T_1]$, where $T_1$ is the lifespan given by
$$
T_1 
=
\frac{c_0}{(1+\|a\|_{\mathcal{A}^{\sigma_0}}+M_{\sigma}(T_0))^d}.
$$
Moreover, $T_1\ge T_0$, as a consequence of \eqref{bound.1}.
Thus, we obtain an extension of $v$ over the interval $[T_0, 2T_0]$ that satisfies
\begin{equation}
\label{acq-mKdVm.eq.2}
M_\sigma(2T_0)
\le
e^{-2\lambda T_0}M_\sigma(T_0)
+ 
C_1\big(\sigma^\theta M_\sigma(T_0) +\sigma \|a\|_{\mathcal{A}^{\sigma_0}}\big)M_\sigma(T_0),
\end{equation}
as established by the almost conserved quantity \eqref{acq-mKdVm.eq} (see Remark \ref{acq-exp}).

Furthermore, using \eqref{bound.1} and the choice in \eqref{cond.1}, inequality \eqref{acq-mKdVm.eq.2} implies that
$$
M_\sigma(2T_0)\le 
M_{\sigma_0}(0)
\qquad \mathrm{and\; consequently}\qquad
\|v(2T_0)\|_{H^{\sigma,0}}\le \|v_0\|_{H^{\sigma_0,0}},
$$
which enables us to apply Theorem \ref{lwp-mKdVm} once more, yielding a solution over $[2T_0,3T_0]$.

Proceeding inductively, this construction extends $v$ over each interval $[kT_0, (k+1)T_0]$, $k=0,1,\ldots$,  covering the entire time domain and satisfying 
\begin{equation}
\label{bound.sol.global}
M_\sigma(kT_0)\le 
M_{\sigma_0}(0)
\qquad \mathrm{and\; consequently}\qquad
\|v(kT_0)\|_{H^{\sigma,0}}\le \|v_0\|_{H^{\sigma_0,0}}.
\end{equation}
By setting $\sigma_1$ as the minimum determined in \eqref{cond.1}, we have the desired time extension for $v$ as in \eqref{const.radius} for any $\tilde{\sigma}_0\le \sigma_1$.

The decay \eqref{exp.decay} is obtained by interpolation between the following two bounds
\begin{equation}
\label{interpolation}
\|v(t)\|_{L^2}
\le 
e^{-\lambda t}\|v_0\|_{L^2}
\quad\text{and}\quad
\|v(t)\|_{H^{\sigma_1,0}}
\le
 \|v_0\|_{H^{\sigma_0,0}},
\end{equation}
which are immediate consequences of \eqref{L2-norm.v-exp} and \eqref{bound.sol.global}, respectively.
In fact, from \eqref{interpolation} we have
$$
\|v(t)\|_{H^{\frac{\sigma_1}{2}},0}
\le
(\|v(t)\|_{L^2}\|v(t)\|_{H^{\sigma_1,0}})^\frac 12
\le
e^{-\frac{\lambda t}{2}}(\|v_0\|_{L^2} \|v_0\|_{H^{\sigma_0,0}})^\frac 12,
$$
which finishes the proof of \eqref{exp.decay} by considering $\tilde{\sigma}_0=\sigma_1/2$ and $C=(\|v_0\|_{L^2} \|v_0\|_{H^{\sigma_0,0}})^\frac 12$.
The proof of Theorem \ref{global-mKdVm} is now complete.
\end{proof}

\appendix

\section{ Fixed lower bound for the radius of analyticity for a coupled system of mKdV-type equations with damping}

In this appendix, we note that the main results established for the mKdVm equation \eqref{mKdVm-IVP} regarding the fixed lower bound for the evolution of the radius of analyticity can be obtained similarly for the following IVP for a coupled system of mKdV-type equations with damping
\begin{equation}\label{system.mKdV-IVP}
\left\{
\begin{array}{l}
\p_t w_1+ \p_x^3w_1+\mu \p_x(w_1w_2^2) +a_1(x)w_1 = 0, \quad x,t\in\rr, \\
\p_t w_2+ \alpha\p_x^3w_2+\mu \p_x(w_1^2w_2) +a_2(x)w_2 = 0, \\
w_1(x,0)=w_{1,0}(x), \quad w_2(x,0)=w_{2,0}(x),
\end{array}
\right.
\end{equation}
where $w_1$ and $w_2$ are real-valued functions, $0<\alpha< 1$ and $\mu=\pm 1$.
This system without damping, which means $a_1=a_2=0$, was also studied in the author's previous work \cite{FP24}, where a decay rate of the evolution of the radius of analyticity was obtained to be $cT^{-1}$.

We assume that the damping functions $a_1$ and $a_2$ satisfy the same conditions {\bf (A1)-(A3)}, with the following modified conditions for damping effect 
$$
a_1(x)\ge \lambda_1>0 \;\;\text{and}\;\;a_2(x)\ge \lambda_2>0,\;\; \text{for all } x\in\rr.
$$
Also, the Bourgain spaces considered in this case are $Z^{\sigma,s,b}$ and $Z^{\sigma,s,b}_\alpha$ with their respective norms defined as follows
\begin{equation*}
\begin{split}
\|w\|_{Z^{\sigma,s,b}} 
= 
\big\|\cosh(\sigma \xi)\la\xi\ra^s\la\tau-\xi^3\ra^b \widehat{w}\big\|_{L^2_{\xi}L^2_{\tau}} 
\;\;\;\text{and}\;\;\;
\|w\|_{Z_\alpha^{\sigma,s,b}} 
= 
\big\|\cosh(\sigma \xi)\la\xi\ra^s\la\tau-\alpha\xi^3\ra^b \widehat{w}\big\|_{L^2_{\xi}L^2_{\tau}}.
\end{split}
\end{equation*}
With this framework, the local well-posedness result for the IVP \eqref{system.mKdV-IVP} is stated as follows.

\begin{theorem}
\label{lwp-system.mKdV}
Let $\sigma_0 >0$, $a_1(x)$ and $a_2(x)$ satisfy the conditions {\bf (A1)}-{\bf (A3)}. 
For each $(w_{1,0},w_{2,0})\in H^{\sigma_0,0}(\rr)\times H^{\sigma_0,0}(\rr)$ there exist $b> 1/2$ and a time
$$
T_0
=
\frac{c_0}{(1+\max\{ \|a_{1}\|_{\mathcal{A}^{\sigma_0}}, \|a_{2}\|_{\mathcal{A}^{\sigma_0}}\} +\max\{\|w_{1,0}\|_{H^{\sigma_0,0}}, \|w_{2,0}\|_{H^{\sigma_0,0}}\})^d},
$$
where $c_0>0$ and $d>1$,
such that the IVP \eqref{system.mKdV-IVP} admits a unique solution
$$
(w_1,w_2)\in C([-T_0,T_0]; H^{\sigma_0,0}(\rr)) \times C([-T_0,T_0]; H^{\sigma_0,0}(\rr)) \cap Z^{\sigma_0,0,b}_{T_0}\times  Z^{\sigma_0,0,b}_{\alpha,T_0}$$
satisfying
\begin{equation}
\label{sol.bound-sys}
\sup\limits_{|t|\le T_0}\Big(\max\{\|w_1(t)\|_{H^{\sigma_0,0}}, \|w_2(t)\|_{H^{\sigma_0,0}}\}\Big)
\le
c\max\{\|w_{1,0}\|_{H^{\sigma_0,0}}, \|w_{2,0}\|_{H^{\sigma_0,0}}\}.
\end{equation}
Furthermore, the data-to-solution map is continuous.
\end{theorem}

The proof of Theorem \ref{lwp-system.mKdV} is analogous to that of Theorem \ref{lwp-mKdVm} with trivial  modifications to address the system case. The main difference lies in the fact that we use the following trilinear estimates in the case of the IVP \eqref{system.mKdV-IVP}, which is an immediate consequence of the trilinear estimates proved in \cite{CP} (for more details of the proof we refer to Lemma $2.2$ in \cite{FH}).

\begin{lemma}
Let $\sigma \ge0$ and $0<\alpha <1$. If $s> -1/2$, then there exist $b$ and $b'$ with $1/2< b'\le b < 1$ such that the following estimates hold for all $w_1\in Z^{\sigma,s,b'}$ and $w_2\in Z_\alpha^{\sigma,s,b'}$
\begin{equation*}
\begin{split}
\|\p_x(w_1 w_2^2)\|_{Z^{\sigma,s,b-1}}
&\lesssim
\|w_1\|_{Z^{\sigma,s,b'}}\|w_2\|^2_{Z_\alpha^{\sigma,s,b'}}\\
\|\p_x(w_1^2 w_2)\|_{Z_\alpha^{\sigma,s,b-1}}
&\lesssim
\|w_1\|^2_{Z^{\sigma,s,b'}}\|w_2\|_{Z_\alpha^{\sigma,s,b'}}.
\end{split}
\end{equation*}
\end{lemma}

Similarly, a global existence result analogous to Theorem \ref{global-mKdVm} can be established for the coupled system \eqref{system.mKdV-IVP}. Under the same conditions on the initial data \( (w_{1,0}, w_{2,0}) \in H^{\sigma_0,0}(\mathbb{R}) \times H^{\sigma_0,0}(\mathbb{R}) \) and damping terms \( a_1 \), \( a_2 \in H^{\sigma_0,0}(\rr)\), the solution \( (w_1, w_2) \) can be extended globally in time. Moreover, this solution exhibits decay properties, with bounds similar to those obtained for the  mKdVm equation \eqref{mKdVm-IVP}. This is precisely stated in the following theorem.

\begin{theorem}\label{global-system.mKdV}
Let $\sigma_0>0$, $w_{1,0}, w_{2,0}\in H^{\sigma_0,0}(\rr)$, $a_1, a_2\in H^{\sigma_0,0}(\rr)$ satisfy {\bf (A1)}-{\bf (A3)}, and $w_1,w_2\in C([-T_0,T_0];H^{\sigma_0,0}(\rr))$ be the local solution to the IVP \eqref{system.mKdV-IVP} guaranteed by Theorem \ref{lwp-system.mKdV}. Then, there is a constant radius \(\tilde{\sigma_0}=\tilde{\sigma_0}(\|w_{1,0}\|_{H^{\sigma_0,0}}, \|w_{2,0}\|_{H^{\sigma_0,0}}, \sigma_0, a_1,a_2)>0,\) such that the solution $(w_1,w_2)$ can be extended globally in time 
and
$$
w_1(t),w_2(t) \in H^{\sigma(t),0}(\rr),
\quad\text{with}\quad
\sigma(t) \ge \tilde{\sigma}_0>0, \;\;\text{ for all } t\ge 0.
$$
Additionally, the following bound holds
$$
\max\{\|w_1(t)\|_{H^{\tilde{\sigma}_0,0}}, \|w_2(t)\|_{H^{\tilde{\sigma}_0,0}}\}
\le
Ce^{-\frac{\lambda_0 }{2}t}, \;\; \lambda_0=\min\{\lambda_1, \lambda_2\}, \;\;\text{ for all } t\ge 0,
$$
with $C$ being a positive constant depending on $a_1(x)$, $a_2(x)$, $\|w_{0,1}\|_{H^{\sigma_0,0}}$ and $\|w_{0,2}\|_{H^{\sigma_0,0}}$. 
\end{theorem}

The key step in proving Theorem \ref{global-system.mKdV} is to establish an almost conserved quantity for the system, analogous to \eqref{acq-mKdVm.eq} for the mKdVm equation. The precise formulation of this almost conserved quantity is provided in the following theorem, whose proof follows by an analogous argument used in the proof of Theorem \ref{acq-mKdVm}.

\begin{theorem}
\label{acq-system.mKdV}
Let $0<\sigma\le \sigma_0$, $w_{1,0},w_{2,0}\in H^{\sigma,0}(\rr)$, $a_1,a_2\in H^{\sigma_0,0}(\rr)$ satisfying {\bf (A1)}-{\bf (A3)}   and $w_1,w_2\in C([-T_0,T_0];H^{\sigma,0}(\rr))$ be the local solution to the IVP \eqref{system.mKdV-IVP} guaranteed by the local well-posedness result presented in Theorem \ref{lwp-system.mKdV}. Then, for any $\theta\in [0,1/2)$ there is a constant $C_1>0$ such that for all $t\in [0,T_0]$
\begin{equation}
\label{acq-system.mKdV.eq}
N_\sigma(t)
\le
e^{-2\lambda_0 t}N_\sigma(0)
+ 
C_1\big(\sigma^\theta N_\sigma(0) 
+
\sigma  \max\{\|a_1\|_{\mathcal{A}^{\sigma_0}},\|a_2\|_{\mathcal{A}^{\sigma_0}}\}\big)N_\sigma(0),
\end{equation}
where $N_\sigma(t)= \|w_1(t)\|_{H^{\sigma,0}}^2 +\|w_2(t)\|^2_{H^{\sigma,0}}$ and
$\lambda_0=\min\{\lambda_1,\lambda_2\}.$
\end{theorem}

We therefore conclude that the methods and results developed for the mKdVm equation can be applied effectively to the system \eqref{system.mKdV-IVP} to establish both the local and global results.


\vskip 0.3cm
\noindent{\bf Acknowledgements.} 
The first author acknowledges support from FAPESP, Brazil (\#2021/04999-9).
The second  author would like to thank FAPESP Brazil for financial support under grant (\#2024/10613-4) and the School of Mathematics, University of Birmingham, UK, for hospitality where a part of this work was developed. 
The authors would also like to thank the anonymous referee whose comments helped immensely in improving the original manuscript. \\


\noindent
{\bf Conflict of interest statement.} 
On behalf of all authors, the corresponding  author states that there is no conflict of interest.\\

\noindent 
{\bf Data availability statement.} 
The datasets generated and/or analyzed during the current study are available from the corresponding author on reasonable request.




\begin{thebibliography}{99}

\bibitem{BFH1}
R. Barostichi, R. Figueira and A. Himonas 
\textit{Well-posedness of the good Boussinesq equation in analytic Gevrey spaces and time regularity} 
J. Differential Equations \textbf{267} (2019) 3181--3198.

\bibitem{BFH}
R. Barostichi, R. Figueira and A. Himonas 
\textit{The modified KdV equation with higher dispersion in Sobolev and analytic spaces on the line,} 
J. Evol. Equ. (2021) 2213--2237.

\bibitem{BGK} J. L. Bona, Z. Gruji\'c and H. Kalisch, 
\textit{Algebraic lower bounds for the uniform radius of spatial analyticity for the generalized KdV equation},
Ann Inst. H. Poincar\'e C Anal Non Linéaire {\bfseries 22} No. 6 (2005) 783--797.

\bibitem{CP} X. Carvajal and M. Panthee, 
{\em Sharp well-posedness for a coupled system of mKdV-type equations}. 
J. Evol. Equ. {\bf 19}  No. 4, (2019) 1167--1197.

\bibitem{CKSTT} J. Colliander, M. Keel, G. Staffilani, H. Takaoka, and T. Tao, 
{ \em Sharp global well-posedness for KdV and modified KdV on $\rr$ and $\mathbb{T}$}, J. Amer. Math. Soc. {\bf 16} No. 3 (2003) 705--749.

\bibitem{CP-12} A. Corcho, M. Panthee, {\em Global well-posedness for a coupled modified KdV system}, Bull Braz Math Soc, New Series {\bf 43} (2012) 27--57.


\bibitem{D} R. C. Davidson,
{\em Methods in Nonlinear Plasma Theory}, Academic Press, New York, (1972).

\bibitem{DMT} T. T. Dufera, S. Mebrate, and A. Tesfahun,
\textit{On the persistence of spatial analyticity for the Beam equation},
J. Math. Anal. Appl. {\bfseries 509} (2022) 126001.

\bibitem{LG} L. G. Farah, 
{\em Global rough solutions to the critical generalized KdV equation,} J. Differential Equations {\bf 249} (2010) 1968--1985.


\bibitem{FH} R. O. Figueira and A. A. Himonas, \textit{Lower bounds on the radius of analyticity for a system of modified KdV equations}, J. Math. Anal. Appl. \textbf{497}, No. 2, (2021) 124917.

\bibitem{FP23} R. O. Figueira and M. Panthee, {\em Decay of the radius of spatial analyticity for the modified KdV equation and the nonlinear Schr\"odinger equation with third order dispersion}, NoDEA Nonlinear Differential Equations Appl.
(2024) 31--68

\bibitem{FP24} R. O. Figueira and M. Panthee, {\em New lower bounds for the radius of analyticity for the mKdV equation and a system of mKdV-type equations}, J. Evol. Equ. {\bf 24} No. 42 (2024).

\bibitem{GTB} T. Getachew, A. Tesfahun and B. Belayneh
\textit{On the persistence of spatial analyticity for generalized KdV equation with higher order dispersion},
Math. Nachr. {\bf 297} No. 5 (2024) 1737--1748.


\bibitem{GK-1} Z. Gruji\'c, H. Kalisch; 
{\em Local well-posedness of the generalized Korteweg-de Vries equation in spaces of analytic functions,} Differential and Integral Equations, {\bf 15} (2002) 1325--1334.

\bibitem{Axel-1}  A. Gr\"unrock, 
{\em A bilinear Airy-estimate with application to gKdV-3}, Differential and Integral Equations {\bf 18} (2005) 133--1339.

\bibitem{HK} Hasegawa, A. and Kodama, Y.
{\em Solitons in Optical Communications}, Clarendon Press, Oxford, (1995).

\bibitem{J} R. S. Johnson,
{\em  A Modern Introduction to the Mathematical Theory of Water Waves}, Cambridge University Press, Cambridge,  (1997).

\bibitem{Ka} Y. Katznelson, 
{\em An Introduction to Harmonic Analysis}, Cambridge University Press, Cambridge,  (1968).

\bibitem{KPV93}
 C. E. Kenig, G. Ponce and L. Vega,
{\em Well-posedness and scattering results for the generalized Korteweg-de Vries equation via the contraction principle}, Comm. Pure Appl. Math. \textbf{46} (1993) 527--620.

\bibitem{KPV91} C. E. Kenig, G. Ponce, L. Vega, 
{\em Oscillatory integrals and regularity of dispersive equations}, Indiana Univ. Math. J. {\bf 40} (1991) 33--69.

\bibitem{KPV91-1} C. E. Kenig, G. Ponce, L. Vega, 
{\em Well-Posedness of the Initial Value Problem for the Korteweg-de Vries Equation}, J. Amer. Math. Soc. {\bf 4}  (1991) 323--347. 

\bibitem{KdV} D. J. Korteweg, G.  de Vries 
{\em On the change of form of long waves advancing in a rectangular canal, and on a new type of long stationary waves}, Phil. Mag. {\bf 39} No. 5 (1895) 422--443.

\bibitem{M} R. M. Miura, 
{\em The Korteweg-de Vries Equation: A Survey of Results}, SIAM Rev. {\bf 18} No. 3 (1976) 412--459.

\bibitem{OBZBMB}
S. Otmani, A. Bouharou, K. Zennir, K. Bouhali, A. Moumen, M. Bouye. 
{\em On the study the radius of analyticity for Korteweg-de-Vries type systems with a weakly damping},
 AIMS Math. {\bf 9} No. 10 (2024) 28341--28360.

\bibitem{SS} 
S. Selberg and D. O. Silva, 
{\em Lower bounds on the radius of a spatial Analyticity for the KdV equation,}
Ann. Henri Poincar\'e {\bf 18 }(2017) 1009--1023.

\bibitem{ST} S. Selberg and A. Tesfahun,  {\em On the radius of spatial analyticity for the 1d Dirac-Klein-Gordon equations}, J. Differential Equations {\bf 259} (2015) 4732--4744.

\bibitem{Tao-b} 
T. Tao, 
{\em Nonlinear Dispersive Equations-Local and Global Analysis,} 
Published for the Conference Board of the Mathematical Sciences, Washington, DC; by the American Mathematical Society, Providence, (2006).

\bibitem{Tao} T. Tao, 
{\em Multilinear weighted convolution of $L^2$ functions and applications to nonlinear dispersive equations,} 
Published for the Conference Board of the Mathematical Sciences, Washington, DC; by the Amer. J.  math {\bf 123} (2001) 839--908.


\bibitem{T} 
A. Tesfahun, 
{\em On the radius of spatial analyticity for cubic nonlinear Schrödinger equations,}
J. Differential Equations {\bf 263} No. 11 (2017) 7496--7512.

\bibitem{TTB} E. Tegegn, A. Tesfahun and B. Belayneh, {\em Lower bounds on the radius of spatial analyticity of solution for KdV-BBM type equations},  NoDEA Nonlinear Differential Equations Appl., (2023) 30--47.

\bibitem{Wa} M. Wang, {\em Nondecreasing analytic radius for the KdV equation with a weakly
damping},  Nonlinear Anal. {\bf 215} (2022) 112653.

\bibitem{W} G. B. Whitham,
{\em Linear and Nonlinear Waves,} Wiley, New York, (1974).
\end{thebibliography}
\end{document}